\documentclass[11pt]{article}
\usepackage{latexsym,color,amsmath,amsthm,amssymb,amscd,amsfonts}

\newtheorem*{remark*}{Remark}

\def\R{{\mathbb R}}

\def\upddots{\mathinner{\mkern 1mu\raise 1pt \hbox{.}\mkern 2mu
\mkern 2mu \raise 4pt\hbox{.}\mkern 1mu \raise 7pt\vbox {\kern 7
pt\hbox{.}}} }

\newcommand{\spn}{{Sp_{2n}(\F)}}
\newcommand{\gsppn}{{GSp^{^+}_{2n}(\F)}}

\newcommand{\gspn}{{GSp_{2n}(\F)}}

\newcommand{\mspn}{{\overline{Sp_{2n}(\F)}}}

\newcommand{\tspn}{T_{2n}(\F)}
\newcommand{\mtspn}{\overline{T_{2n}(\F)}}
\newcommand{\tgspn}{T'_{2n}(\F)}
\newcommand{\mtgspn}{\overline{T'_{2n}(\F)}}
\newcommand{\tgsppn}{T^{^+}_{2n}(\F)}
\newcommand{\mtgsppn}{\overline{T^{^+}_{2n}(\F)}}
\newcommand{\bspn}{B_{2n}(\F)}
\newcommand{\bgspn}{B'_{2n}(\F)}

\newcommand{\mbspn}{{\overline{B_{2n}(\F)}}}
\newcommand{\mbgspn}{{\overline{B'_{2n}(\F)}}}
\newcommand{\mbgsppn}{{\overline{B^{^+}_{2n}(\F)}}}

\newcommand{\zspn}{{{N_{2n}}}(\F)}

\newcommand{\mzspn}{\overline{{{N_{2n}}}(\F)}}

\newcommand{\tgln}{T_{GL_n}(\F)}

\newcommand{\zgln}{N_{_{GL_n}}(\F)}

\newcommand{\mgspn}{\overline{GSp_{2n}(\F)}}

\newcommand{\F}{F}

\newcommand{\Of}{\mathbb O_{\F}}
\newcommand{\Pf}{\mathbb P_{\F}}

\newcommand{\spx}{\spn}

\newcommand{\f}{_{_\F}}

\newcommand{\gspx}{{\gspn}}

\newcommand{\mgpspn}{\overline{GSp^{^+}_{2n}(\F)}}
\newcommand{\mgsppn}{\overline{GSp^{^+}_{2n}(\F)}}

\newcommand{\kgspn}{{GSp_{2n}(\Of)}}
\newcommand{\kspn}{{Sp_{2n}(\Of)}}
\newcommand{\mkgspn}{{K^\eta_{2n}(\F)}}
\newcommand{\mkspn}{{K_{2n}(\F)}}
\newcommand{\mkgspnp}{{K^+_{2n}(\F)}}
\newcommand{\mkgspnm}{{K^-_{2n}(\F)}}

\newcommand{\mpb}{{\overline{B^+_{2n}(\F)}}}
\newcommand{\mb}{{\overline{B_{2n}(\F)}}}

\newcommand{\half}{\frac{1}{2}}

\newcommand{\ab} {|}
\newcommand{\gl}{{GL_n(\F)}}
\newcommand{\gln}{{GL_n(\F)}}

\newcommand{\C}{\mathbb C}

\def\>{\rangle}
\def\<{\langle}

\newtheorem{lem}{Lemma}[section]
\newtheorem{thm}{Theorem}[section]
\newtheorem{cor}{Corollary}[section]

\numberwithin{equation}{section}

\newcommand{\msl}{\overline{SL_2(\F)}}

\def\dotunion{
\def\dotunionD{\bigcup\kern-9pt\cdot\kern5pt}
\def\dotunionT{\bigcup\kern-7.5pt\cdot\kern3.5pt}
\mathop{\mathchoice{\dotunionD}{\dotunionT}{}{}}} 
\begin{document}
\author{Dani Szpruch}
\title {Symmetric genuine Spherical Whittaker functions on $\mgspn$} \maketitle
\begin{abstract}
Let $\F$ be a p-adic field of odd residual characteristic. Let $\mgspn$ and $\mspn$ be the metaplectic double covers of the general symplectic group and the symplectic group attached to the $2n$ dimensional symplectic space over $\F$. Let $\sigma$ be a genuine, possibly reducible, unramified principal series representation of $\mgspn$. In these notes we give an explicit formulas for a spanning set for the space of Spherical Whittaker functions attached to $\sigma$. For odd $n$, and generically for even $n$, this spanning set is a basis. The significant property of this set is that each of its elements is unchanged under the action of the Weyl group of $\mspn$. If $n$ is odd then each element in the set has an equivariant property that generalizes the uniqueness result of Gelbart,  Howe and Piatetski-Shapiro proven in \cite{GHP}. Using this symmetric set, we construct a family of reducible genuine unramified principal series representations which have more then one generic constituent. This family contains all the reducible genuine unramified principal series representations induced from a unitary data and exists only for $n$ even.
\end{abstract}

\section{Introduction}
Let $\F$ be a p-adic field and let $\Of$ be its ring of integers. Let $\mspn$ be the metaplectic double cover of $\spn$ and let $\mgspn= \mspn \rtimes \F^*$ be the metaplectic double cover of $\gspn$. Let $\mtspn$ and $\mtgspn$ be the inverse images inside $\mgspn$ of $\tspn$ and $\tgspn$, the maximal tours of $\spn$ and $\gspn$ respectively. While $\mtspn$ is commutative, $\mtgspn$ is not. The isomorphism class of a genuine smooth admissible irreducible representation of $\mtgspn$ is determined by its central character. Let $I(\omega)$ be a, possibly reducible,  genuine principal series representation of $\mgspn$ parabolically induced from a representation of $\mtgspn$ whose central character is $\omega$ . Its Whittaker model is not unique. In fact
$$dim \Bigl(Hom_{\mgspn}\bigl(\sigma, Ind^{\mgspn}_{\mzspn}\theta\bigr)\Bigr)=[{\F^*}:{\F^*}^2].$$
Here $\theta$ is a non-degenerate genuine character of $\mzspn$, the inverse image of the maximal unipotent radical of $\spn$. We fix a natural basis, $B(\omega,\theta)$, for this space.

Suppose now that $\F$ has odd residual characteristic. In this case, $\mgspn$ splits over $\kgspn$. While for odd $n$ the embedding of $\kgspn$ inside $\mgspn$ is essentially unique, there are two non conjugate embeddings if $n$ is even. Let $\mkgspn$ be an image of $\kgspn$ inside $\mgspn$. Assume that $I(\omega)$ is spherical, i.e., contains a one dimensional invariant $\mkgspn$ space. Note that this property is independent of the particular embedding of $\kgspn$. In these notes we provide an explicit formulas for $W(\omega,\theta,\eta)$, the set of images of the spherical vector under $B(\omega,\theta)$, i.e., a spanning set for the $\mkgspn$ invariant subspace inside the subspace of $Ind^{\mgspn}_{\overline {N_{2n}(\F)}}\theta$ generated by the images of $\sigma$, see Theorem \ref{main} for the explicit formulas. For odd $n$, and generically for even $n$, $W(\omega,\theta,\eta)$ is a basis. Let $\mtgsppn$ be the centralizer of $\mtspn$ inside $\mtgspn$ and let $E(\omega)$ be the set of extensions of $\omega$ to $E(\omega)$. Each element of $W(\omega,\theta,\eta)$ corresponds to one of the elements of $E(\omega)$ . The important property of the spanning set constructed here is that it consists of symmetric functions, i.e, each element is preserved under the action of the Weyl group of $\mspn$ on the characters of $\mtgsppn$. Given the global functional equation satisfied by Eisenstein series, see Theorem IV.1.10 of [21], this property is meaningful. In the cases where a Weyl element induces a non trivial permutation on $E(\omega)$, $W(\omega,\theta,\eta)$ becomes linearly dependent. We give exact description of these cases which occurs only if $n$ is even. As an application, we construct a family of reducible genuine unramified principal series representations of $\mgspn$ which includes all the reducible genuine unramified principal series representations induced from unitary data. The representations in this family are a direct sum of two non-isomorphic $\theta$-generic $\mgspn$ irreducible submodules, each contains an invariant vector with respect to a different embedding of $\kgspn$.

Our argument in this paper is based on the relation between representation theory of $\mgspn$ and $\mspn$ given in \cite{Sz 5} along with the work of Bump, Friedberg and Hoffstein on the Spherical Whittaker function for $\mspn$. It is the uniqueness of Whittaker model for $\mspn$ which is ultimately responsible for our main result. We shall now outline the principles that enable our results. Although we assume the $\F$ is p-adic field of odd residual characteristic  most of what follows in this introduction, i.e., everything outside the context of unramified representations, applies  2-adic fields also. In fact, it applies to the field of real numbers as well.

For an element $g \in \gspn$ we shall denote by $\lambda(g)$ the similitude character of $g$ and by  $\overline{g}$ its inverse image inside $\mgspn$. We shall define $\lambda(\overline{g})=\lambda(g)$. Let $\gsppn$ be the subgroup of $\gspn$ which consists of elements whose similitude character lies in ${\F^*}^2$. For any subset $\overline{H'}$ of $\mgspn$ denote by $\overline{H^{^+}(\F)}$ and $\overline{H}$ its intersection with $\mgsppn$ and $\mspn$ respectively. We note that the intermediate group $\mgsppn$ introduced here is the same group that appears in the context of Theta correspondence between $\mgspn$ and similitude orthogonal groups attached to certain odd dimensional symmetric forms, see \cite{Ro}. Let $M'(\F)$ be a Levi subgroup of $\gspn$. It is a fact that $$\overline{M^{^+}(\F)}=Z\bigl(\mgsppn \bigr) \overline{M(\F)}.$$ Here $Z\bigl(\mgsppn \bigr)=\overline{\F^* I_{2n}}$ is the center of $\mgsppn$. Furthermore $$Z\bigl(\mgsppn \bigr) \cap \overline{M(\F)}=\overline{\pm I_{2n}}.$$ Thus, given the representation theory of
$\overline{M(\F)}$ the representation theory of $\overline{M^{^+}(\F)}$ is trivial. In particular, the uniqueness of Whittaker model for irreducible smooth admissible genuine representations of $\overline{M^{^+}(\F)}$ follows immediately from the corresponding uniqueness for  $\overline{M(\F)}$ proven in \cite{Sz 1}. On the other hand, assume that $M'(\F)$ is isomorphic to
$$GL_{n_1}(\F) \times GL_{n_2}(\F) \ldots \times GL_{n_r}(\F) \times GSp_{2n_{r+1}}(\F),$$
where at least one of the numbers $n_1,\ldots n_{r+1}$ is odd. This condition is satisfied by any Levi subgroup of $\gspn$, including $\gspn$ itself, provided that $n$ is odd and it is satisfied by $T'_{2n}(\F)$, regardless of the parity of $n$. Let $\pi$ be a genuine irreducible smooth admissible representation of $\overline{M'(\F)}$. It was shown in \cite{Sz 5} that the restriction of $\pi$ to $\overline{M^{^+}(\F)}$ is a multiplicity free direct sum of $[\F^*: {\F^*}^2]$ summands. This is equivalent to the fact that $\overline{M} / \overline{M^{^+}(\F)}$ acts freely and transitively on the set of irreducible $\overline{M^{^+}(\F)}$ modules that appears in $\pi$. If $\pi
\! ^{^+}$ is one of these summands, then $$\pi \simeq Ind_{\overline{M^{^+}(\F)}}^{\overline{M'(\F)}} \pi \! ^{^+}\simeq Ind_{\overline{M^{^+}(\F)}}^{\overline{M'(\F)}} {^m \! (\pi \! ^{^+})},$$
where ${^m \! (\pi \! ^{^+})}$ is a conjugation of $\pi \! ^{^+}$ by an element $m$ of $\overline{M'(\F)}$. Suppose now that $\pi$ is $\theta$-generic. It follows that at least one of the summands in the restriction of $\pi$ to $\overline{M^{^+}(\F)}$ is $\theta$-generic. In fact the dimension of the space of Whittaker functionals on $\pi$ is the number of $\theta$-generic summands. Furthermore, there is a natural way to describe the space of $\theta$-Whittaker functionals on  $\pi$ by means of the unique (up to a constant) $\theta$-Whittaker functionals on one of the $\theta$-generic summands. This also gives a description of the Whittaker functionals on a parabolic induction from $\pi$ to $\mgspn$. We shall describe this in the case $M'(\F)=T'_{2n}(\F).$ In this case $\mtgsppn$ is a maximal Abelian subgroups of $\mtgspn$ and the construction above agrees with the theory of genuine representations of metaplectic tori, see Theorem 3 of \cite{Mc}. It should be noted that  $\mtgsppn$ is not an analog to Kazhdan-Patterson standard maximal Abelian subgroup, i.e., it is not the centralizer of $\mtgspn \cap \mkgspn$ inside $\mtgspn$. The fact that $\mtgsppn$ is the centralizer of $\mtspn$ inside $\mtgspn$, rather then just being a maximal Abelian subgroup, is a crucial property. Thus, any genuine principle series representation of $\mgspn$ may be realized as

$$I\bigl((\chi \boxtimes \xi)_\psi \bigr)=Ind_{\mbgsppn}^{\mgspn} (\chi \boxtimes \xi)_\psi$$ where $(\chi \boxtimes \xi)_\psi$ is a genuine character of $\mtgsppn$, extended to $\mbgsppn= \mtgsppn \ltimes \zspn$, see Section \ref{rep sec} for details. We note here that $\chi$ is character of $\tspn$. For $c \in \F^* / {F^*}^2$, let $I^{c}\bigl((\chi \boxtimes \xi)_\psi \bigr) \subseteq I\bigl((\chi \boxtimes \xi)_\psi \bigr)$ be the subspace of functions whose support is contained in the open set of elements whose similitude character lies in  $c {\F^*}^2$. Obviously, $I^{c}\bigl((\chi \boxtimes \xi)_\psi \bigr)$ is a $\mgsppn$ invariant space and
$$I\bigl((\chi \boxtimes \xi)_\psi \bigr)= \! \! \! \! \! \bigoplus_{c \in \F^* / {F^*}^2} \! \! \!  \! \! I^{c}\bigl((\chi \boxtimes \xi)_\psi \bigr).$$ The crucial fact is that as an $\mspn$ module
$$I^{c}\bigl((\chi \boxtimes \xi)_\psi \bigr) \simeq I\bigl((\chi \eta_c)_\psi \bigr),$$
where $$I\bigl((\chi \eta_c)_\psi \bigr)=Ind_{\mb}^\mspn(\chi \eta_c)_\psi.$$
Here $\chi \eta_c$ is the quadratic twist of $\chi$ attached to $c$ and $(\chi \eta_c)_\psi$ is a genuine character of $\mtspn$, see Section \ref{rep sec} for details. It is now clear that any Whittaker functional on $I\bigl((\chi \boxtimes \xi)_\psi \bigr)$ which vanish on $$I_c \bigl((\chi \boxtimes \xi)_\psi \bigr)= \! \! \! \! \!  \bigoplus_{c \neq d \in \F^* / {F^*}^2} \! \! \! \! \!  I^{d}\bigl((\chi \boxtimes \xi)_\psi \bigr)$$ defines a  Whittaker functional on   $I\bigl((\chi \eta_c)_\psi \bigr)$. Consequently, the space of Whittaker functionals on  $I\bigl((\chi \boxtimes \xi)_\psi \bigr)$ which vanish on $I_c \bigl((\chi \boxtimes \xi)_\psi \bigr)$ is one dimensional. For $c=1$, as a non-trivial element in this space, we can take (the analytic continuation of)
$$\lambda_{(\chi \boxtimes \xi)_\psi,1,\theta}(f)=\int_{N_{2n}(\F)} f(J_{2n}u) \theta^{-1}(u) \, du. $$
where $J_{2n} \in \mspn$ is a representative of the long Weyl element. Suppose now that $t \in \mtgspn$ is such that $\lambda(t)=c$, the map $f \mapsto f_t$ defined by $f_t(g)=f(tg)$ is a $\mgspn$ isomorphism between $I\bigl((\chi \boxtimes \xi)_\psi \bigr)$ and $I \bigl(^t (\chi \boxtimes \xi)_\psi  \bigr)$, where $^t (\chi \boxtimes \xi)_\psi $ is a conjugation of $(\chi \boxtimes \xi)_\psi$ by $t$. This isomorphism maps $\lambda_{(\chi \boxtimes \xi)_\psi,1,\theta}$ to a $\theta_t$ Whittaker functional on $I \bigl(^t (\chi \boxtimes \xi)_\psi \bigr)$ which vanish on $I_c \bigl(^t (\chi \boxtimes \xi)_\psi \bigr) $ . Here $\theta_t$ stands for $n \mapsto \theta({t^{-1}nt})$. At the same time, this map defines a $\theta$ Whittaker functional on $I\bigl((\chi \boxtimes \xi)_\psi \bigr)$ which vanish on $I_c \bigl((\chi \boxtimes \xi)_\psi \bigr)$ namely
$$\lambda_{(\chi \boxtimes \xi)_\psi,t,\theta_t}(f)=\int_{N_{2n}(\F)} f(J_{2n}ut) \theta_t^{-1}(u) \, du. $$

Hence, we obtain a basis for the space of $\theta$-Whittaker functionals on $I\bigl((\chi \boxtimes \xi)_\psi \bigr)$. Note that this basis is defined only up to permutations as it depends on choosing an initial realization and only up normalization since the set of isomorphism above is depended on the choice of representatives of $\mtgspn / \mtgsppn$.

Fix now $W_{2n}$, a set of inverse images inside $\mspn$ of a set of representatives the Weyl group of $\mspn$. $W_{2n}$ is also a set of representatives of $N \bigl(\mtgspn \bigr) / \mtgspn$ where $N \bigl(\mtgspn \bigr)$ stands for the normalizer  of $\mtgspn$ inside $\mgspn$ . Define $N_{2n}^-(\F)$ to be the unipotent radical opposite to $\zspn$. For $w \in W_{2n}$, let $$A_w:I\bigl((\chi \boxtimes \xi)_\psi \bigr) \rightarrow I\bigl((  {^w \! \chi} \boxtimes \xi)_\psi \bigr)$$ be the standard intertwining operator defined by the meromorphic continuation of
$$\bigl(A_w(f)\bigr)(g)=\int_{\zspn \cap(w  N_{2n}^-(\F) w^{-1})} f(wng) \, dn. $$
Note that $A_w$ maps $I^c\bigl((\chi \boxtimes \xi)_\psi \bigr)$ to $I^c\bigl((  ^w \!   \chi \boxtimes \xi)_\psi \bigr)$. Hence,  $\lambda_{(\chi \boxtimes \xi)_\psi,c,\theta}$ and $\lambda_{( ^w \! \chi \boxtimes \xi )_\psi,c,\theta} \circ A_w$ are proportional. This explains why the functional equation satisfied by these Jacquet integrals is diagonal. In fact, for odd $n$, these Jacquet integrals  have a uniqueness property which generalizes the uniqueness of  Whittaker model discussed in \cite{GHP}, see Corollary \ref{n odd stru} and Theorem \ref{Final}.

On a recent paper by Chinta and  Offen, \cite{CO}, the authors used the method of Casselman and Shalika, \cite{CS}, to construct a spanning set of Spherical Whittaker functions for covers of the general linear group. It is clear from their construction that a symmetric spanning set may be constructed by introducing Jacquet integrals whose functional equation is diagonal and then by properly normalizing each integral. To normalize the Jacquet integrals here we use the work of  Bump, Friedberg and Hoffstein. In \cite{BFH} these authors computed the normalization factor for one $\mtspn$ orbit of Whittaker characters. The above exposition explains why is it necessary to extend their result to the other orbits.

For the $n=1$ our result coincides with the recipe given in Section 7 of \cite{CO}. In fact, our work gives a natural explanation to the diagonalization given there. It would be interesting to provide a similar explanation to the higher rank cases studied in \cite{CO}. In the $n=1$ case, $\chi$ is a character of $\F^*$, $\eta_c$ is the quadratic a character of $\F^*$ attached to $c$ and there exists only one non-trivial Weyl element. By comparing the results of \cite{Sz 2} and \cite{Sweet} one can show that the proportion factors in the functional equations mentioned above are
$$\gamma_\F(\psi^{-1})q^{e/2}\frac
{\gamma(\chi^2,0,\psi_2)}{\gamma(\chi \cdot \eta_c,\half,\psi)},$$
where $\gamma_\F(\psi^{-1})$ is the non-normalized Weil index attached to $\psi^{-1}$, $q$ is the size of the residual field, $e=[\Of:2 \Of]$, and $\gamma(\cdot,\cdot,\cdot)$ is the Tate gamma factor. These factors should be thought of as an analog to Shahidi Local coefficients, \cite{Sha 1}, defined in a context where uniqueness of Whittaker model fails. The fact that these coefficients are distinct shows that the diagonalization above is unique up to permutations and normalization.

It should be noted that the proofs given in this paper do not relay on \cite{Sz 5} and do not make explicit use of the uniqueness principles discussed above: Wherever needed we used the theory of principle series representations of covering groups already existing in the literature.  The paper is organized as follows: In Sections \ref{gen not} and \ref{groups} we define some general notation and collect some information on $\mgspn$. In Section \ref{rep sec} we construct the genuine principal series representations and their Whittaker maps. We also discuss in some details the Weyl group action and prove some irreducibility results. In Section \ref{BFH sec} we give an extension of the results of \cite{BFH}. This extension is used in Section \ref{last} where we prove our main result. Finally, in Section \ref{app} we apply the main result to construct the family of reducible genuine unramified principal series representations mentioned above.

I would like thank Freydoon Shahidi for constantly explaining to me his thoughts on the subject matter. In fact, the application given in Section \ref{app} is a part of an on going joint work in which we extend the results of \cite{KS} to $\mgspn$. I would also like to thank Gordan Savin and Omer Offen for sharing with me their insights on the subject.

\section{General notation} \label{gen not}
Let $\F$ be a p-adic field, let $\Of$ be the ring of integers of
$\F$, let $\Pf$ be its maximal ideal and let $\pi$ be a generator of $\Pf$. Let $q$ be the
cardinality of the residue field ${\Of}/{\Pf}$ and let
$\ab \cdot \ab$ be the absolute value on $\F$ normalized in the
usual way. For $a \in \F^*$ we denote  its order by $ord(a)$. Thus,
$ord(a)=-log_q(\ab a \ab).$

Let $(\cdot,\cdot)\f$ be the quadratic Hilbert symbol of $\F$. It is a non-degenerate symmetric bilinear form on $\F^* /{\F^*}^2$. For $a \in \F^*$ we define $\eta_a$ to be the quadratic character of $\F^*$ attached to $a$, that is, $$\eta_a(b)=(a,b)\f.$$
Let $\psi$ be a non-trivial additive character of $\F$. We define its conductor to be the smallest integer such that $\psi$ is trivial on $\Pf^n$. We say that $\psi$ is normalized if its conductor is 0. For $a\in \F^*$ define  $$\psi_a(x)=\psi(ax).$$ It is also a non-trivial additive character of $\F$. For $a\in \F^*$ let
$\gamma_\psi(a) \in \C^{1}$ be the normalized Weil factor associated with
the character of second degree of $\F$ given by $x \mapsto
\psi_a(x^2)$ (see Theorem 2 of Section 14 of \cite{Weil}). It is
known that $\gamma_{\psi}$ is a forth root of unity and that $\gamma_{\psi}\bigl({\F^*}^2\bigr)=1$. Also,
\begin{equation} \label{gammaprop}
\gamma_{\psi}(ab)=\gamma_{\psi}(a)\gamma_{\psi}(b)(a,b)_{\F}, \, \,  \, \,   \gamma_{\psi_b}=\eta_b \cdot \gamma_\psi  ,
\, \,  \, \,  \gamma_{\psi}^{-1}=\gamma_{\psi^{-1}}, \, \,  \, \, \gamma_\psi({\F^*}^2)=1.
\end{equation}
For a diagonal matrix $t$ inside $\gln$ we shall define
$$\gamma_\psi(t)=\gamma_\psi \bigl(det(t)\bigr), \, \, \eta_a(t)=\eta_a\bigl(det(t)\bigr).$$

In Sections \ref{split compact}, \ref{BFH sec}, \ref{last} and \ref{app}, we assume that $\F$ has an odd residual characteristic. In this case, $(\Of^*,\Of^*)\f=1$. Moreover, it follows from Lemma 1.5 of \cite{Sz 2} that if the conductor of $\psi$ is even then $\gamma_\psi(\Of^*)=1$. In the odd residual characteristic case we fix a set of representatives of ${\F^*}/{\F^*}^2$
of the form $\{1,u_{_0},\pi,\pi u_{_0}\},$ where $u_{_0}$ is a non-square element in $\Of^*$, fixed once and for all. The non-degeneracy of the Hilbert symbol implies in this case that \begin{equation} \label{unram quad}\eta_{u_0}(a)=(-1)^{ord(a)} \end{equation}
For p-adic fields of odd residual characteristic, there exists exactly two quadratic characters of $\Of^*$. The non-trivial one is \begin{equation} \label{non trivial eta} u \mapsto\eta_\pi(u)=(u,\pi)\f=\begin{cases}1  & u \in {{\Of^*}\!^2}, \\ -1  &  u \notin {{\Of^*}\!^2} \end{cases}.\end{equation}
Any of the two quadratic characters $\eta$ of $\Of^*$ extends uniquely to a quadratic character of ${\Of^*}{\F^*}^2$. We shall continue to denote this character by $\eta$.

Let $G$ be a group. We denote its center by $Z(G)$. For $h,g \in G$ we denote $h^g=g^{-1}hg$. Let $H$ be a subgroup of $G$. Let $\sigma$ be a representation of $H$. If $g \in G$ normalizes $H$ we denote by $^g\sigma$ the representation of $H$ defined by $h\mapsto \sigma(h^g)$.

\section{Groups} \label{groups}
\subsection{Linear groups}
Let $\gspn$ be the general symplectic group attached to the 2n dimensional symplectic space over $\F$. We shall realize $\gspn$ as the group $$\{g \in GL_{2n}(\F) \mid gJ_{2n}g^t=\lambda(g)J_{2n} \},$$
where
$$J_{2n}=\begin{pmatrix} _{0} & _{I_{n}}\\_{-I_{n} } & _{0}
\end{pmatrix}$$ and $\lambda(g)\in \F^*$ is the similitude factor of $g$. The similitude map $g\mapsto \lambda(g)$ is a rational character of $\gspn$. The kernel of the similitude map is the symplectic group, $\spn$. $\F^*$ is embedded in $\gspx$ via $$\lambda
\mapsto i(\lambda)=\begin{pmatrix} _{I_n} & _{0}\\_{0} & _{\lambda
I_n}\end{pmatrix}.$$ Using this embedding we define an action of
$\F^*$ on $\spx$: $$(g,\lambda) \mapsto
g^{i(\lambda)}.$$ Let $\F^* \ltimes
\spx$ be  the semi-direct product corresponding to this action.
For $g \in \gspx$ define $$g_{_1}=i\bigl(\lambda^{-1}(g)\bigr)g \in \! \spx.$$
 The map $$g \mapsto \bigl(\lambda(g),g_{_1} \bigr)$$ is an isomorphism between
$\gspx$ and $\F^* \ltimes \spx.$
We define $\gsppn$ to be the subgroup of $\gspn$ which consists of elements
whose similitude factor lies in ${\F^*}^2$. $\gsppn$ is a normal subgroup of $\gspn$ which contains $\spn$ . Clearly, $$[\gspn:\gsppn]=[{\F^*}:{\F^*}^2] < \infty.$$
For any $H'$ of $\gspn$ we denote by $H^+$ and $H$ its intersection with $\gsppn$ and $\spn$ respectively.

Let $\zgln$ be the group of upper triangular unipotent matrices in
$\gl$ and let $\zspn$ be the following maximal
unipotent subgroup of $\spn$:
$$\Bigl\{\begin{pmatrix} _{z} & _{b}\\_{0} & _{^t \!{z} \!^ {-1}}
\end{pmatrix}  \mid z\in \zgln, b \in Mat_{n \times n}
(\F), \, b^t=z^{-1}b z^t \Bigr \}.$$ We normalize the Haar measure on $\zspn$ in the usual way. Let $\psi$ be a non trivial character of $\F$. A character $\theta$
of $\zspn$ is called non-degenerate if
$$\theta(z)=\psi \bigr(a_nz_{(n,2n)}+\Sigma_{i=1}^{n-1} a_i z_{(i,i+1)}\bigl),$$
where $a_1,a_2,\ldots,a_{n}$ are elements of $\F^*$. We say that $\theta$ is normalized if $\psi_{a_i}$ is normalized for any $1 \leq 1 \leq n$. We say that $\theta$ agree with $\psi$ on the long root if $a_n=1$.

Let $\tgln$ be the group of diagonal elements inside $\gln$ and let $$\tgspn= \{[t,y]=diag(t,yt^{-1}) \mid y\ \in \F^*, \, \, t\in \tgln \}$$ be the subgroup of diagonal matrices
inside $\gspn$. Note that $\lambda\bigl([t,y]\bigr)=y$. Denote $$\bgspn=\tgspn \ltimes \zspn.$$
It is a Borel subgroup of $\gspn$. Let $\delta$ be the modular function on $\bgspn$, i.e., the ratio between left and right Haar measures.

\subsection{Rao's cocycle and metaplectic groups} \label{des rao}
Let $\mspn$ be the unique non-trivial two-fold cover of $\spn$. The action of $\F^*$ on $\spn$ lifts uniquely to an action of $\F^*$ on $\mspn$, see page 36 of \cite{MVW}. Using this lift we define  $$\mgspn \simeq \F^* \ltimes \mspn,$$ the unique two-fold cover of $\gspn$ which contains $\mspn$.

We shall realize $\mgspn$ as the set  $\mgspn=\gspx \times \{ \pm 1 \}$ equipped with the multiplication law
\begin{equation} \label{metaplectic structure}(g_{1},\epsilon_{1})(g_{2},\epsilon_{2})=
\bigl( g_{1}g_{2},\epsilon_{1}\epsilon_{2}\widetilde{c}(g_{1},g_{2})\bigr)
.\end{equation} Here $$\widetilde{c}:\gspn \times \gspn \rightarrow \{\pm 1\}$$ is the cocycle constructed in Section 2B of \cite{Sz 1} and studied in Section 2 of \cite{Sz 5}.
$\widetilde{c}$ is an extension of Rao's cocycle
$$c:\spn \times \spn \rightarrow \{\pm 1\}$$
constructed in \cite{R}. Hence, the inverse image of $\spn$ inside $\mgspn$ is a realization of $\mspn$. For any subset $H$ of $\gspn$ we denote by $\overline{H}$ its inverse image inside $\mgspn$. Let $H$ be a subgroup of $\gspn$. A representation $\pi$ of $\overline{H}$ is called genuine if it does not factor through the projection map on $H$. Thus, a representation of $\overline{H}$ with a central charter is genuine if and only $\pi(I_{2n},-1)=-Id$ . For $(g,\epsilon) \in \mgspn$ we define $\lambda(g,\epsilon)=\lambda(g).$
\begin{lem} \label{coc prop} The following hold:\\
(1) $\widetilde{c}\bigl( [t,y],[t',y'] \bigr)=\bigl(det(t),y'det(t')\bigr)\f$.\\
(2) $\widetilde{c} \bigl(i(\F^*),\gspn \bigr)x=1$.\\
(3) $\bigl(g,\epsilon \bigr)^{(aI_{2n},\epsilon')}=\bigl(g, \epsilon(\lambda(g),a^n)\bigr)\f$ and $\bigl(aI_{2n},\epsilon \bigr)^{(g,\epsilon')}=\bigl(a, \epsilon(\lambda(g),a^n)\bigr)$.\\
(4) $Z \bigl(\mgspn \bigr)=\begin{cases}\overline{\F^*I_{2n}}  & n \, is \, even, \\ \overline{{\F^*}^2I_{2n}}  &  n \, is \, odd \end{cases}.$\\
(5) $Z \bigl( \mgpspn \bigr)=\overline{\F^*I_{2n}}$.\\
(6) $Z \bigl( \mtgspn \bigr)=\{ \bigl([t,y],\epsilon \bigr) \mid y,\det(t) \in {\F^*}^2 \}$.\\
(7) $\mtspn$ is an abelian group, $\mtgsppn$ is the centralizer of $\mtspn$ inside $\mtgspn$ and it is a maximal Abelian subgroup of $\mtgspn$. Furthermore
\begin{equation} \label{tp st} \mtgsppn=\mtspn Z \bigl( \mgpspn \bigr) \end{equation}
and  $$Z \bigl( \mgpspn \bigr) \cap \mtspn=\overline{\pm I_{2n}}.$$
(8) $\zspn$ embeds into $\mspn$ via the trivial section. $\mtgspn$ normalizes $\bigl(\zspn,1 \bigr)$.\\
\end{lem}
\begin{proof}
(1) and (2) and (3) follow immediately from the cocycle formulas given in pages 456 and 460 of \cite{Sz 1}. Since $Z(\overline{H}) \subseteq \overline{Z(H)}$ for any subgroup $H$ of $\gspn$, (4) and (5) follow from (3). (6) and (7) follow from (1). We now prove (8): Since $\zspn \subseteq \spn$, the fact that $N_{2n}(\F)$ is embedded inside $\mgspn$ via the trivial section is a property of Rao cocycle, see Corollary 5.5 of \cite{R}. This Corollary also implies that $\mtspn$ normalizes $\bigl(\zspn,1 \bigr)$. Thus, to finish the proof of (8) one only need to show that $\overline{i(\F^*)}$ normalizes $(N_{2n},1)$. This fact may be verified directly from the cocycle formula. All the assertions given in this lemma are proven in more details in Section 2 of \cite{Sz 5}.
\end{proof}
From the last part of this Lemma it follows that if $\theta$ is a character of $\zspn$ then $(z,\epsilon)\mapsto \epsilon \theta(z)$ defines a genuine character of $\mzspn$. We shall continue to denote this character by $\theta$.

\subsection{ Splittings of maximal compact subgroups} \label{split compact}
In this subsection we assume that $\F$ is a p-adic field of odd residual characteristic. In this case, $\mspn$ splits over $\kspn$ and $\mgspn$ splits over $\kgspn$, see Theorem 2 of \cite{Mc}. Furthermore, there exists a unique map $$\iota_{2n}:\kspn \rightarrow \{\pm 1 \},$$ such that $$k \mapsto \kappa_{2n}(k)=\bigl(k,\iota_{2n}(k) \bigr)$$ is an embedding of $\kspn$ inside $\mspn$. It is known that $\iota_{2n}$ is identically 1 on $\bspn \cap  \kspn $, see section 1.4 of \cite{Sz 4}. We shall denote the image of $\kspn$ under this embedding by $\mkspn$. The splitting of $\mgspn$ over $\kgspn$ is not unique. Let $\eta$ be one of the two quadratic characters of $\Of^*$. Define $$\iota^\eta_{2n}:\kgspn \rightarrow \{\pm 1 \},$$ to be the following extension of  $\iota_{2n}$\begin{equation} \label{map one}\iota^\eta_{2n}(k)= \eta\bigl(\lambda(k)\bigr) \iota_{2n} \bigl(k_{_1} \bigr),\end{equation}
Note that these two maps are identically 1 on  $\tgsppn \cap  \kgspn$.
\begin{lem} \label{2 split}
There are exactly two maps  $$\iota'_{2n}:\kgspn \rightarrow \{\pm 1 \},$$ such that $k \mapsto \bigl(k,\iota'_{2n}(k) \bigr)$ is an embedding of $\kgspn$ inside $\mgspn$: These are the two maps defined in \eqref{map one}.  These two embeddings  are conjugated by an element of $\mgspn$ if and only if $n$ is odd.
\end{lem}

\begin{proof}

The restriction of $\iota'_{2n}$ to $\kspn$ must agree with $\iota_{2n}$. Thus $$\iota'_{2n}(k)= \iota'_{2n}\bigl(i(\lambda(k)\bigr) \iota_{2n}(k_{_1})\widetilde{c}\bigl(i(\lambda(k),k_{_1}\bigr).$$
Part 2 of Lemma \ref{coc prop} implies now that the restriction of $\iota'_{2n}$ to $i(\Of^*)$ must be a quadratic character and that
 $$\iota'_{2n}(k)= \iota'_{2n}\bigl(i(\lambda(k)\bigr) \iota_{2n}(k_{_1}).$$
This shows that $\iota'_{2n}$ must be one of the two maps given in \eqref{map one}. On the other hand, if $\iota'_{2n}$ is a splitting map, so does $\iota'_{2n} \cdot (\eta_\pi \! \circ \! \lambda)$.

Let $\iota'_{2n}$ be any of the two splittings given in \eqref{map one}. Suppose that $n$ is odd. By part 3 of Lemma \ref{coc prop}, $$\bigl(k,\iota'_{2n}(k)\bigr)^{(\pi I_{2n},1)}=\bigl(k,\iota'_{2n}(k) \eta_{\pi}(\lambda(k))\bigr).$$

Hence, if $n$ is odd, the two embeddings constructed here are conjugated. It remains to show that if $n$ is even then these two are not conjugated. From the Cartan decomposition it follows that the  normalizer of $\kgspn$ inside $\gspn$ is $Z\bigr(\gspn \bigl) \kgspn.$ Since the inner conjugation map of $\mgspn$, $$\sigma \mapsto \sigma^{(g,\epsilon)}$$ is independent of $\epsilon$,  it follows that if the two embeddings mentioned above are conjugated then a conjugating element must lie in $\overline{Z \bigl(\gspn \bigr)}$. However, if $n$ is even then by Lemma \ref{coc prop}, $ \overline{Z \bigl(\gspn \bigr)}=Z \bigl(\mgspn \bigr).$
\end{proof}
From this point we assume that $\kgspn$ is embedded in $\mgspn$ via $$k \mapsto \bigl(k,\iota^\eta_{2n}(k) \bigr),$$
where $\eta$ is one of the two characters above. We shall denote the image of $\kgspn$ under this embedding by $\mkgspn$.

\section{Representations} \label{rep sec}
\subsection{Genuine Principal series representations}

\begin{lem} \label{ind rep}Let $\chi$ be a character of the diagonal subgroup inside $\gln$ and let $\xi$ be a character of $\F^*$. \\
(1) The map \begin{equation} \label{mtspn char} \bigl([t,1],\epsilon \bigr) \mapsto \chi_\psi\bigl([t,1],\epsilon \bigr)=\epsilon \chi(t)\gamma_\psi(t)\end{equation} defines a genuine character of $\mtspn$ and all genuine characters of $\mtspn$ have this form. Also,
\begin{equation} \label{psi cha} \chi_{\psi_a}=(\chi \cdot \eta_a)_\psi. \end{equation}
(2) The map $$(aI_{2n},\epsilon) \mapsto \xi_\psi(aI_{2n},\epsilon)=\epsilon \xi(a)\gamma_\psi(a^n)$$
defines a genuine character of $Z \bigl( \mgpspn \bigr)$ and all genuine characters of $Z \bigl( \mgpspn \bigr)$ have this form.\\
(3) Assume in addition that $\xi(-1)=\chi(-I_n)$ or equivalently that $\xi_\psi$ and $\chi_\psi$ agree on  $Z \bigl( \mgpspn \bigr) \cap \mtspn.$
Then, we may define a genuine character of $\mtgsppn$ by
$$h=zg \mapsto (\chi \boxtimes \xi)_\psi (zg)=\chi_\psi(g)\xi_\psi(z).$$
Here $g \in \mtspn$ and $z \in Z \bigl( \mgpspn \bigr)$. All the genuine characters of $\mtgsppn$ have this form.\\
(4) $(\chi \boxtimes \xi)_\psi$ and $(\chi' \boxtimes \xi')_\psi$ have the same restriction to $Z\bigl(\mtgspn \bigr)$ if and only if they are conjugated by an element of $\mtgspn$ in which case \begin{equation} \label{model cha}(\chi' \boxtimes \xi')_\psi=(\chi \cdot \eta_c \boxtimes \xi\cdot \eta_c^n)_\psi=^{(i(c),1)} \! \! \bigl((\chi \boxtimes \xi)_\psi \bigr) \end{equation}
For some $c\in\F^*$.
\end{lem}
\begin{proof} The fact that \eqref{mtspn char} defines a genuine character of $\mtspn$ follows from the first part of Lemma \ref{coc prop} and from \eqref{gammaprop}. Since the product of any two genuine characters of $\mtspn$ is a non-genuine character of $\mtspn$ and since any non-genuine character of $\mtspn$ may be viewed as a character $\tspn$, the first two assertions in (1) follow. \eqref{psi cha} also follows now from \eqref{gammaprop}. (2) is proven by similar arguments. (3) follows from (1), (2) and from the seventh part of Lemma \ref{coc prop}. The equality between the middle and right wings of \eqref{model cha} follows from the first part of  Lemma \ref{coc prop}. The fact the any two genuine characters of $\mtgsppn$ which satisfy the relation given in \eqref{model cha} have the same restriction to $Z\bigl(\mtgspn \bigr)$  follows from (6) and (7) of
Lemma \ref{coc prop}. Note now that $[\mtgsppn:Z\bigl(\mtgspn \bigr)]=[{\F^*}:{\F^*}^2]$ and that
$$\{(i(c),1) \mid c \in {\F^*}/{\F^*}^2 \}$$ is a set of representatives for $\mtgspn / \mtgsppn$. The proof of (4) is finished once we note that $c \notin {\F^*}^2$ implies that
$$(\chi \boxtimes \xi)_\psi \neq (\chi \cdot \eta_c \boxtimes \xi\cdot \eta_c^n)_\psi.$$

\end{proof}
Let $\omega$ be a genuine character of $Z\bigl(\mtgspn \bigr)$ we shall denote by $E(\omega)$ the set of its extensions to $\mtgsppn$. The group $\mtgspn$ is not Abelian. Its smooth admissible irreducible representations have the following form:
\begin{lem} \label{mc me} The isomorphism class of a smooth admissible genuine irreducible representation of $\mtgspn$ is determined by its central character. Any smooth admissible genuine irreducible representation of $\mtgspn$ is $[\F^*:{\F^*}^2]$ dimensional. It may be realized as an induction from a genuine character of $\mtgsppn$. If $\varphi$ and $\varphi'$ are two genuine characters of
$\mtgsppn$ then the corresponding inductions on $\mtgspn$ are isomorphic if and only if $\varphi, \varphi' \in E(\omega)$, where $\omega$ is a genuine character of $Z\bigl(\mtgspn \bigr)$.
\end{lem}
\begin{proof} This is a variation of Stone-von Neumann Theorem, see Theorem 3 of \cite{Mc}. The crucial facts here are that $\mtgspn$ is a two step nilpotent group and that $\mtgsppn$ is a maximal Abelian subgroup.
\end{proof}

We shall denote by $\tau(\omega)$ the isomorphism class of irreducible admissible genuine representation of $\mtgspn$ whose central character is $\omega$. We denote its realization mentioned in the last Theorem by $$i \bigl((\chi \boxtimes \xi)_\psi \bigr)=Ind_{\mtgsppn} ^{\mtgspn}(\chi \boxtimes \xi)_\psi.$$

Due to part $8$ of Lemma \ref{coc prop} we shall regard, as in the linear case, each character (irreducible smooth admissible representation) of $\mtspn$ $ \bigl( \mtgspn \bigr)$ as a character (representation) of $\mbspn$ $ \bigl( \mbgspn \bigr)$ by extending it trivially on $(N_{2n}(\F),1)$. We shall always assume that a character (representation) of $\mbspn$ $\bigl( \mbgspn \bigr)$ has this form.
Genuine principal series representation of $\mspn$ $ \bigl( \mgspn \bigr)$ is a representation parabolically induced from a genuine character (irreducible smooth admissible genuine representation) of $\mbspn$ $\bigl( \mbgspn \bigr)$. We shall assume that the induction is normalized. We denote
$$I(\chi_\psi)=Ind^\mspn_\mbspn \chi _{\psi}.$$
For a genuine character $\omega$ of $Z \bigl( \mtgspn \bigr)$ we denote by $I(\omega)$ the isomorphism class of genuine principal series representation of $\mgspn$ induced from $\tau(\omega)$. By Lemma \ref{mc me}, $I(\omega)$ may be realized as $$I \bigl((\chi \boxtimes \xi)_\psi \bigr)=Ind_\mbgsppn ^{\mgspn}(\chi \boxtimes \xi)_\psi,$$
where $(\chi \boxtimes \xi)_\psi$ is any of the $[F^*:{F^*}^2]$ elements of  $E(\omega)$.

\subsection{Weyl group action}

Let $W_{2n}$ be the Weyl group of $\spn$. We choose representatives of $W_{2n}$ inside $\kspn$ and we take their images inside $\mkspn$. We shall continue do denote the outcome by $W_{2n}$. By (1.18) of Section 1.3 of \cite{Sz 4},
\begin{equation} \label{w on t} (t,1)^w=(t^w,1) \end{equation}
for any $t \in \tspn$ and $w \in W_{2n}$. In fact, \eqref{tp st} implies that \eqref{w on t} holds for all $t \in \tgsppn$. Thus,
$$^w (\chi_\psi)= \bigl({^w \! \chi})_\psi$$
and
\begin{equation} \label{to compare} ^w (\chi \boxtimes \xi)_\psi= ({^w \!\chi} \boxtimes \xi)_\psi\end{equation}
For any $w \in W_{2n}$. Here $^w \! \chi$ is defined via a similar action of $W_{2n}$ on $\tspn$ and its characters. Note that $W_{2n}$ also acts on the set of genuine charters of $Z\bigl(\mtgspn \bigr)$ and hence acts on isomorphism classes of genuine irreducible admissible representations of $\mtspn$. Clearly, $E(^w \! \omega)={^w \!E(\omega)}$ for any $w \in W_{2n}$. Given $\omega$, a genuine character of $Z\bigl(\mtgspn \bigr)$, fix  $(\chi \boxtimes \xi)_\psi \in E(\omega)$ and define $R(\omega)$ to be the the following subgroup of $\F^* /{\F^*}^2$:

$$R(\omega)=\{ c\in \F^* /{\F^*}^2 \mid \exists \omega \in W_{2n}. \, ^{(i(c),1)} \! (\chi \boxtimes \xi)_\psi={^{w} \! (\chi \boxtimes \xi)_\psi} \}.$$
Clearly, $R(\omega)$ is well defined, i.e., does not depend on the particular choice of $(\chi \boxtimes \xi)_\psi \in E(\omega)$. $R(\omega)$ is non-trivial if an only if there exits $w \in W_{2n}$ which fixes $\omega$ and induces a non-trivial permutation on $E(\omega)$.
\begin{lem} \label{act on E} Let $\omega$ be a genuine character of $Z\bigl(\mtgspn \bigr)$\\
(1) If $n$ is odd and $^w \! \omega = \omega$ for some  $w \in W_{2n}$ then $w$ preserves the elements of $E(\omega)$ pointwise. In particular, $R(\omega)$ is trivial for any $\omega$.\\
(2) If $n$ is even then there exists a genuine character $\omega_{_0}$ of $Z\bigl(\mtgspn \bigr)$ such that $R(\omega_{_0})$ is non-trivial. Furthermore, up to conjugation by $w \in W$, an element of $E(\omega_{_0})$ has the from
$(\chi_{_0} \boxtimes \xi)_\psi$  where
\begin{equation} \label{r not tri} \chi_{_0}[diag(t_n,t_{n-1},\ldots,t_1),1]=\chi_1(t_1t_2)\chi_2(t_3,t_4)\ldots \chi_{\frac n 2}(t_{n-1}t_n)\eta_c(t_2t_4\ldots t_{n}). \end{equation}
Here $\chi_1,...\chi_{\frac n 2}$ are characters of $F^*$ and $c$ is a non-square element in $\F^*$.

\end{lem}
\begin{proof} The first assertion follows from a comparison of \eqref{model cha} and \eqref{to compare}. Next, if $n$ is even, define
 \begin{equation} \label{special w} w_{_0}=diag(w',w',\ldots,w') \in W_{2n}, \end{equation}
where $w'=\begin{pmatrix} _{0} &
_{1}\\_{1} & _{0} \end{pmatrix}.$
By a straight forward computation $$^{w_{_0}} \! \bigl(\chi_{_0} \boxtimes \xi)_\psi =( \chi_{_0} \eta_c \boxtimes \xi)_\psi.$$
$W_{2n}$ is generated by simple reflections. Denote these reflections by $w_{i,i+1}$ where $1 \leq i \leq n-1$ and by $w_n$ where $w_{i,i+1}$ acts on $[diag(t_n,t_{n-1},\ldots,t_1),1])$ by switching
$t_i$ and $t_{i+1}$ and where $w_{1}$ acts on $[diag(t_n,t_{n-1},\ldots,t_1),1] $ by inverting $t_1$. Thus, if $\chi'$
defined by $$diag(t_n,t_{n-1},\ldots,t_1)=\prod_{i=1}^{n}\chi'_i(t_i)$$
is not Weyl conjugated to $\chi_0$ then there exists $1 \leq i \leq n$ such that $\chi'_i \neq \chi_j^{\pm 1}$ for any $1 \leq j \leq n$. This shows that
$\chi_\psi'$ is not Weyl conjugated to any of its quadratic twists.

\end{proof}

\begin{thm} \label{basic red} Let $\omega$ be a genuine character of $Z\bigl(\mtgspn \bigr)$. Fix $(\chi \boxtimes \xi)_\psi \in E(\omega)$.\\
(1) $I(\omega)$ is irreducible if and only if $R(\omega)$ is trivial and $I(\chi_\psi)$ is an irreducible $\mspn$ module. \\
(2) If $I(\chi_\psi)$ is an irreducible $\mspn$ module and $R(\omega)$ has two elements then  $I(\omega)$ breaks into a direct sum of two $\mgspn$ irreducible modules.\\
(3) If the restriction of $\omega$ to $Z\bigl(\mtgspn \bigr) \cap \mspn$ is unitary and $\R(\omega)$ has two elements then  $I(\omega)$ breaks into a direct sum of two $\mgspn$ irreducible modules.\\
\end{thm}
\begin{proof}
Using induction by stages we may realize $I(\omega)$ as
$$Ind_{\mgsppn}^{\mgspn} I'\bigl ((\! \chi\boxtimes \xi)_\psi\bigr),$$
where \begin{equation} \label{i prime} I'\bigl( (\! \chi \boxtimes \xi)_\psi\bigr)=Ind^{\mgsppn}_{\mbgsppn}(\! \chi \boxtimes \xi)_\psi. \end{equation}
Since $\mgsppn$ is a normal subgroup of finite index inside $\mgspn$ it follows from Clifford theory that $I(\omega)$ is irreducible if and only if
$I'\bigl( (\! \chi \boxtimes \xi)_\psi\bigr)$ is irreducible and it is not isomorphic to any of its conjugations by elements of $\mgspn$ which lies outside $\mgsppn$.  Note that since $\mgsppn=\mspn Z\bigl(\mgsppn \bigr)$, it follows that $I'\bigl( (\! \chi \boxtimes \xi)_\psi\bigr)$ is an irreducible $\mgsppn$ module if and only if $I \bigl( \chi_\psi \bigr)$ is an irreducible $\mspn$ module. If $S$ is a set of representatives of  ${\F^*}/{\F^*}^2$ then $\bigl(i(S),1\bigr)$ is a set of representatives of $\mgspn / \mgsppn$ and that for $c\in \F^*$ we have
$$^{(i(c),1)} \! I'\bigl( (\! \chi \boxtimes \xi)_\psi\bigr) \simeq I'\bigl( ^{(i(c),1)} \!(\! \chi \boxtimes \xi)_\psi \bigr).$$
Suppose now that $I'\bigl( (\! \chi \boxtimes \xi_a)_\psi\bigr)$ is irreducible. If $n$ is odd then by the first part of Lemma \ref{act on E}, $R(\omega)$ is trivial and for $c \notin {\F^*}^2$
$$I'\bigl( (\! \chi \boxtimes \xi)_\psi\bigr) \not \simeq {^{(i(c),1)} \! I'\bigl( (\! \chi \boxtimes \xi)_\psi\bigr)}$$ since by the third part of Lemma \ref{coc prop}, these two representations have different central characters. This proves (1) for odd $n$. \emph{}Suppose now that $n$ is even. Then, by \eqref{model cha}, $I'\bigl( (\! \chi \boxtimes \xi)_\psi\bigr)$ and $I'\bigl( ^{(i(c),1)} \!(\! \chi \boxtimes \xi)_\psi \bigr)$ have the same cental character. Hence,
$$I'\bigl( (\! \chi \boxtimes \xi)_\psi\bigr) \simeq {^{(i(c),1)} \! I'\bigl( (\! \chi \boxtimes \xi)_\psi\bigr)}$$
if and only if \begin{equation} \label{in sec}I \bigl( \chi_\psi \bigr) \simeq I \bigl(  (\chi \eta_c)_\psi \bigr).\end{equation}
Since we assume that  $I \bigl(\chi_\psi \bigr)$ is irreducible it follows, similar to linear groups, that\eqref{in sec} holds if and only if
 $$ (\chi \eta_c)_\psi= (^w \!\chi)_\psi$$
for some $w \in W_{2n}$. This finishes the proof of the first assertion.

We now prove the second assertion. Since we assume that $I(\chi_\psi)$ is irreducible, then, by Clifford theory, the assertion will follow once we show that $A(\omega)$, the commuting algebra of
$$Ind_{\mgsppn}^{\mgspn} I'\bigl (( \chi \boxtimes \xi)_\psi\bigr)$$
is 2 dimensional. By Frobeniuns reciprocity
$$A(\omega) \simeq Hom_{\mgsppn}  \Bigl(Ind_{\mgsppn}^{\mgspn} I'\bigl ((\! \chi \boxtimes \xi)_\psi\bigr),I'\bigl ((\! \chi \boxtimes \xi)_\psi\bigr) \Bigr) \simeq$$
$$\bigoplus_{y \in F^*/{F^*}^2} Hom_{\mgsppn}  \Bigl(Ind^{\mgsppn}_{\mbgsppn}\bigl( (\chi \eta_y \boxtimes \xi)_\psi \bigr),Ind^{\mgsppn}_{\mbgsppn}(\! \chi\boxtimes \xi)_\psi \Bigr) \simeq$$ $$\bigoplus_{y \in Y} Hom_{\mspn} \Bigl( I\bigl(   \chi \eta_y)_\psi\bigr), I \bigl(\chi_\psi \bigr)  \Bigr).$$
By an argument we already used and by definition of $R(\omega)$ we finish. The third assertion follows from the second and from The fact
that, regardless of the parity of $n$, any genuine principal series of $\mspn$ induced from a unitary character is irreducible, see \cite{HM} or Theorem 5.1 of \cite{Sz 4}.

\end{proof}
Note: We focus here on the case where $R(\omega)$ has two elements since, as we shall see in Section \ref{app}, if $\omega$ is unramified then $R(\omega)$ has at most two elements. However it is quite possible, for ramified $\omega$, that $R(\omega)$ will have 4 elements, which is the upper bound if  $\F$ is a p-adic field of odd residual characteristic. Indeed, if in \eqref{r not tri} we assume that $\chi_1^2=\chi_2^2=\ldots=\chi_n^2=\eta_d$ where $d$ is a non square element in $\F^*$ such that $dc^{-1} \not \in {\F^*}^2$ then it is easy to see that since
$^{J_{2n}} \! \bigl(\chi_{_0} \boxtimes \xi)_\psi =( \chi_{_0} \eta_d \boxtimes \xi)_\psi$, it follows that $R(\omega)$ has 4 elements.

\subsection{Whittaker functionals}
Let $\theta$ be  a non-degenerate character of $\zspn$ and let $\sigma$ be a genuine representation of $\mspn$ $ \bigl( \mgspn \bigr)$. A $\theta$-Whittaker model for $\sigma$ is a non-zero image of $\sigma$ inside $Ind_{\mzspn}^{\mspn} \theta$ $ \bigl(Ind_{\mzspn}^{\mgspn} \theta \bigr)$ under an $\mspn$ $ \bigl( \mgspn \bigr)$ map. $\sigma$ is called $\theta$- generic if it has a $\theta$-Whittaker model. Note that if $\theta$ and $\theta'$ are conjugated by an element of $\mtspn$ $ \bigl( \mtgspn \bigr)$ then $\sigma$ is $\theta$-generic if and only if it is $\theta'$-generic. Note that there is only one orbit of non-degenerate characters of $\zspn$ under the action of $\mtgspn$ while there are $[\F^* :{\F^*}^2]$ orbits of non-degenerate characters of $\zspn$ under the action of $\mtspn$. If $y$ varies over a set  of representatives of $\F^* /{\F^*}^2$ then $$\theta_y={^{(i(y),1)} \! \theta}$$ varies over all $\tspn$ orbits of a non-degenerate genuine characters of $\zspn$. Note that we may choose a set of representatives $\F^* /{\F^*}^2$ which consists of elements whose order is 0 or 1. Any genuine principal series representation $\sigma$ is $\theta$-generic. Its $\theta$-Whittaker model is unique, see \cite{BFH} or \cite{Sz 1}.
\begin{lem} For $f \in I(\chi_\psi)$ and $g \in \mspn$ the integral
$$W_f(g)=\int_{N_{2n}(\F)} f\bigl((J_{2n}u,1)g \bigr)\theta^{-1}(u,1) \, du. $$ converges absolutely in some cone inside the set of genuine characters of $\mtspn$ and has an analytic continuation to the full set of genuine characters. The map $f\mapsto W_f$ is the unique (up to a scalar) map from $I(\chi_\psi)$ to its $\theta$-Whittaker model.
\end{lem}
\begin{proof} This is well known and follows from Theorem 5.2 of \cite{BZ}. See  for example Chapter 4 of \cite{BanPhD} or Theorem 6.3 of \cite{Mc}.
\end{proof}
Whittaker models for $\mgspn$ are not unique. For principal series representations we have the following
\begin{lem} \label{basis lemma} Fix $A$, a set of $[\F^*:{\F^*}^2]$ representatives of $\mtgspn / \mtgsppn$. For $f \in I\bigl(\chi \boxtimes \xi)_\psi\bigr), \, g \in \mgspn, \, a \in A$ the integrals

$$\widetilde{W}^a_f(g)=\int_{N_{2n}(\F)} f\bigl(a(J_{2n}u,1)g \bigr)\theta^{-1}(u,1) \, du $$

converges absolutely in some cone inside the set of genuine characters of $\mtgsppn$ and has an analytic continuation to the full set of genuine characters. The set of maps
$$\{f\mapsto \widetilde{W}^a_f \mid a\in A\}$$ form a basis for  $$Hom_{\mgspn}\Bigl(I\bigl((\chi \boxtimes \xi)_\psi\bigr), Ind^{\mgspn}_{\overline {N_{2n}(\F)}}\theta\Bigr).$$
\end{lem}
\begin{proof} This is proven exactly as in Lemmas 1.3.1 and 1.3.2 of \cite{KP}. The idea here if that $$\{f\mapsto f(a) \mid a\in A\}$$ is a basis of the space of functionals on $i \bigl((\chi \boxtimes \xi)_\psi \bigr)$ and that the Jacquet integral is an isomorphism from this space to $Hom_{\mgspn}\bigl(I\bigl(\chi \boxtimes \xi)_\psi\bigr), Ind^{\mgspn}_{\overline {N_{2n}(\F)}}\theta\bigr).$ This theorem may also be deduced from the arguments in Section 5 of \cite{Sz 5}.
\end{proof}

\begin{cor} \label{n odd stru} Keep the notation of Lemma \ref{basis lemma} and assume that $n$ is odd. Denote by $\omega'$ the central character of  $I\bigl((\chi \boxtimes \xi)_\psi\bigr)$. $\omega'$ is the restriction of $(\chi \boxtimes \xi)_\psi$ to $Z \bigl(\mgspn \bigr)=\overline{{F^*}^2I_{2n}}$. Let $\Omega$ be the set of the $[F^*:{\F^*}^2]$ extensions of $\omega'$ to $Z \bigl(\mgsppn \bigr)=\overline{{F^*}I_{2n}}$. For each $\mu \in \Omega$, exactly one of the maps $f\mapsto \widetilde{W}^a_f$ is an element in the one dimensional space

$$Hom_{\mgspn}\Bigl(I\bigl((\chi \boxtimes \xi)_\psi\bigr), Ind^{\mgspn}_{Z \bigl(\mgspn \bigr) \times \overline {N_{2n}(\F)}} \mu \times\theta\Bigr).$$
\end{cor}
\begin{proof}
Let $W$ be the space of $\theta$ Whittaker functionals on $I\bigl((\chi \boxtimes \xi)_\psi\bigr)$, i.e., the set of functionals on $I\bigl((\chi \boxtimes \xi)_\psi\bigr)$ such that $\lambda(\rho(n)f)=\theta(n)\lambda(f)$ for any $n \in \mzspn$. Here $\rho$ stands for right translations. By Frobenius reciprocity $$W \simeq Hom_{\mgspn}\Bigl(I\bigl((\chi \boxtimes \xi)_\psi\bigr), Ind^{\mgspn}_{\overline {N_{2n}(\F)}}\theta\Bigr)$$
and $$\{f\mapsto \widetilde{W}^a_f(I_{2n},1) \mid a\in A\}$$ form a basis for $W$. Clearly $Z \bigl(\mgsppn \bigr)$ acts on $W$. By examining its action on this basis it follows that $W$ decomposes over $Z \bigl(\mgsppn \bigr)$  with multiplicity 1 to $[F^*:{F^*}^2]$ one dimensional spaces and that $\Omega$ is the set of $Z \bigl(\mgsppn \bigr)$ eigen values. The Lemma follows now from Frobenius reciprocity.
\end{proof}
\begin{lem} \label{basis lemma again}
Let $Y$ be a set of representatives of $[\F^*:{\F^*}^2]$. For $f \in I\bigl((\chi \boxtimes \xi)_\psi\bigr)$ and $g \in \mgspn, \, y \in Y$
define $W^y_f(g)$ to be the analytic continuation of $$\int_{N_{2n}(\F)} f\bigl((J_{2n}u,1)(i(y),1)g \bigr)\theta_y^{-1}(u,1) \, du. $$
The set of maps
$$\{f\mapsto W^a_f \mid a\in A\}$$ form a basis for  $$Hom_{\mgspn}\Bigl(I\bigl((\chi \boxtimes \xi)_\psi\bigr), Ind^{\mgspn}_{\overline {N_{2n}(\F)}}\theta\Bigr).$$
\end{lem}
\begin{proof} As a set of representatives of $\mtgspn / \mtgsppn$ we choose $$\{\bigl(diag(yI_{2n},I_{2n}),1 \bigr)\mid y \in Y \}.$$
We note that up to $(I_{2n},\pm 1)$
$$(diag(yI_{2n},I_{2n}),1 \bigr)(J_{2n}u,1)=(J_{2n}u,1)(i(y),1).$$ The lemma follows now from the Lemma \ref{basis lemma} by changing integration variable $u \mapsto u^{(i(y),1)}$.
\end{proof}
\section{Extension of a result of Bump, Friedberg and Hoffstein}  \label{BFH sec}
From this point until the end of this paper, we assume that the residual characteristic of $\F$ is odd. We fix $\psi$, a normalized character of $\F$, and $\theta$, a non-degenerate normalized character of $N_{2n}(\F)$ which agrees with $\psi^{-1}$ on the long root.

A genuine character of $\mtspn$ is called unramified if it is trivial on $\mkspn \cap \mtspn$. Fix $\alpha=(\alpha_1, \alpha_2,\ldots,\alpha_n) \in {\C^*}^n$ and define $\chi_\alpha$ to be a character of $\tspn$ by
$$\chi_\alpha(diag(a_n,a_{n-1},\ldots,a_1, a_n^{-1}, a_{n-1}^{-1},\ldots,  a_1 ^{-1}) =\prod_{i=1}^{n}\alpha_i^{ord(a_i)}.$$
Since $\psi$ has an even conductor, $\gamma_\psi(\Of^*)=1$. Thus, any  unramified genuine character of $\mtspn$ may be written as $({\chi_\alpha})_\psi$ for some $\alpha \in {\C^*}^n$. Note that the action of $W_{2n}$ maps an unramified character of $\mtspn$ to an unramified character. Denote
$$^w({\chi_\alpha})_\psi= ({\chi_{^w \! \alpha}})_\psi.$$
This action induces an action of $W_{2n}$ on the polynomial ring $$R=\C[\alpha_1,\alpha_2,\ldots,\alpha_n,\alpha_1^{-1},\alpha^{-1}_2,\ldots,\alpha^{-1}_n].$$ Let $A$ be the following linear map on $R$:
$$A(p)=\sum_{w\in W_{2n}} (-1)^{length(w)} \, {^w \! p}.$$
Define $$\Delta(\alpha)=A\bigl(\prod_{i=1}^n \alpha_i^i\bigr).$$

If $({\chi_\alpha})_\psi$ is a genuine unramified character of $\mtspn$, then $I \bigl(({\chi_\alpha})_\psi \bigr)$  contains a one dimensional $\mkspn$ fixed subspace.  We denote by $f^0_{({\chi_\alpha})_\psi}$ the normalized Spherical function of $I \bigl(({\chi_\alpha})_\psi \bigr)$, i.e., the unique $\mkspn$ invariant element inside $I \bigl(({\chi_\alpha})_\psi \bigr)$ such that $$f^0_{({\chi_\alpha})_\psi}(I_{2n},1)=1.$$
An image of $f^0_{({\chi_\alpha})_\psi}$ inside the $\theta_y$-Whittaker model of $I \bigl(({\chi_\alpha})_\psi \bigr)$, which is unique up to a scalar is called a Spherical $\theta_y$-Whittaker function attached to $I \bigl(({\chi_\alpha})_\psi \bigr)$. A Spherical $\theta_y$ Whittaker function of $W^0_{\alpha,\psi,y,\theta}$ is called symmetric if
\begin{equation} \label{func eq} W^0_{\alpha,\psi,y,\theta}=W^0_{^w \! \alpha,\psi,y,\theta} \end{equation} for any $w \in W_{2n}$. Following \cite{BFH}, We shall now give an explicit formulas for
these elements where the order of $y$ is either 0 or 1. Define
$$D(\alpha,y)=\prod_{i<j}\bigl((1-q^{-1}\alpha_j\alpha_i^{-1})(1-q^{-1}\alpha_j\alpha_i)\bigr) \begin{cases}\prod_{i=1}^{n} (1+\eta_\pi(-y)q^{-\half}\alpha_i)  & \ab y \ab=1, \\ \prod_{i=1}^{n} (1-q^{-1}\alpha_i^2)  & \ab y \ab=q^{-1} \end{cases} $$
and define
$$W^0_{\alpha,\psi,y,\theta}(g)=D(\alpha,y)^{-1}\int_{N_{2n}(\F)} f^0_{\chi_{\alpha,\psi}}\bigl((J_{_{2n}}n,1)g\bigr) \theta^{-1}_y(n) \, dn$$
\begin{lem} \label{BFH res}
For $t= diag(a_n,a_{n-1},\ldots,a_1) \in \tgln$ denote $$ord(a_i)=k_i.$$ For $y\in \Of^*$ We have
\begin{equation} \label{nor orbits}  W^0_{\alpha,\psi,y,\theta} \bigl([t,1],1 \bigr)=$$ $$ \Delta^{-1}(\alpha)\delta(t)^{\half} \gamma_{\psi}^{-1}(t)\begin{cases}  A\bigl(\prod_{i=1}^n (1-\eta_\pi(-y) q^{\frac{-1}{2}}\alpha_i^{-1}) \alpha_i^{k_i+i}\bigr) & 0\leq k_1 \leq k_2 \leq \ldots \leq k_n \\ 0 & othewise \end{cases}. \end{equation}
For $y$ such that $\ab y \ab=q^{-1}$ we have
\begin{equation} \label{unnor orbits} W^0_{\alpha,\psi,y,\theta}(t,1)=\\ \Delta^{-1}(\alpha)\delta(t)^{\half} \gamma_{\psi}^{-1}(t)\begin{cases}  A\bigl(\prod_{i=1}^n (\alpha_i^{k_i+i}\bigr) & 0\leq k_1 \leq k_2 \leq \ldots \leq k_n \\ 0 & othewise \end{cases}.\end{equation}
In particular, for any $y \in \F^*$ whose order is either 0 or 1, $W^0_{\alpha,\psi,y,\theta}$ is a symmetric Spherical $\theta_{y}$-Whittaker function attached to $I \bigl(({\chi_\alpha})_\psi \bigr)$.
\end{lem}
\begin{proof} We first prove \eqref{nor orbits}. For $y=1$ this is exactly Theorems 1.1 and 1.2 of \cite{BFH}. Note that we denote by $\theta$ what these authors denote by $\theta^{-1}$, see page 384 of \cite{BFH}. However, this difference is canceled since we use $\theta^{-1}$ rather then $\theta$ in the Jacquet integral.  Given that the Whittaker character $\theta_y$ is normalized, the only relation of it with $\psi$ used in the proof of these Theorems is that $\psi_y$, which is the restriction of $\theta_y$ to the long simple root, satisfies $\gamma_\psi=\gamma_{\psi_y}$. Recall that $\gamma_{\psi_y}=\gamma_{\psi}\eta_y$. This explains why the result for any $y \in {\Of^*}\!^2$ follows from Theorem 1.1 and 1.2 of \cite{BFH}. For $y\in \Of^*$ which are not squares one can still use the results in \cite{BFH}: From \eqref{unram quad} and \eqref{psi cha} it follows that
$$ (\chi_\alpha)_\psi=(\chi_{(y,\pi)\alpha})_{\psi_y}=(\chi_{-\alpha})_{\psi_y}, \, \, \, \, \gamma_{\psi_y}(t)=\gamma_{\psi}(t)(-1)^{k_1+k_2+\ldots k_n}.$$
Also,
\begin{equation} \label{to be used again} \Delta(\alpha)={(-1)}^{(1+2+\ldots+n)}\Delta(-\alpha).\end{equation}
This implies that by changing $\theta$ to $\theta_y$, $\gamma_\psi$ to $\gamma_{\psi_y}$ and $\alpha_i$ to $-\alpha_i$ in page 384-386 of \cite{BFH} one obtains the result above.

We shall now explain how to modify the proofs in \cite{BFH} to prove \eqref{unnor orbits}. We start by establishing the absolute convergence of the Jacquet integral in some cone in $\C^n$, its holomorphic extension to $\C^n$ and the functional equation satisfied by $W^0_{\chi_{\alpha,\psi},y,\theta}$, namely \eqref{func eq}. The only modification one needs to do in the proof of Theorem 1.1 of \cite{BFH} is the $\msl$ computation. By a direct computation, similar to one given in page 387 of \cite{BFH}, one proves that if $Re(\alpha)>0$ then for $t=u\pi^{k_1}$, where $u \in \Of^*$,
$$\int_\F f^0_{(\chi_{\alpha_1})_\psi}
\Biggl( \bigl(\begin{pmatrix} _{0} & _{1}\\_{-1} &
_{0}\end{pmatrix}\begin{pmatrix} _{1} & _{x}\\_{0} &
_{1}\end{pmatrix},1\bigr) \bigl(\begin{pmatrix} _{t} & _{0}\\_{0} &
_{t^{-1}}\end{pmatrix},1 \bigr) \Biggr) \psi_\pi^{-1}(x) dx=$$
$$(1-q^{-1}\alpha_1^2) \delta(t)^{\half} \gamma_{\psi}^{-1}(t)\begin{cases}  \frac{(a_1^{k_1+1}-a_1^{-k_1-1})}{\alpha_1-\alpha_1^{-1}}\ & 0\leq k_1  \\ 0 & othewise \end{cases}. $$
Since the restriction of $\theta_{y}$ to the short simple roots are normalized the rest of the argument goes word for word as the proof of Theorem 1.1 of \cite{BFH}. Next, the fact that $W^0_{\alpha,\psi,y,\theta}(t,1)$ vanishes unless $0\leq k_1 \leq k_2 \leq \ldots \leq k_n$ is proven by the same simple argument that is given in page 392 of \cite{BFH}. Indeed, this argument works for any Whittaker character under the assumption that its restriction to any short simple root subgroup is non-trivial on $\Pf^{-1}$ and that its restriction to the long simple root subgroup is non-trivial on $\Pf^{-2}$. The main remaining ingredient in the proof is an expansion of $f^0_{\chi_{\alpha,\psi}}$ as a linear combination of Iwahori fixed vectors. This step is independent of the Whittaker character. Last, since $\theta_y$ is trivial on $N_{2n}(\Of)$ it follows that
$$\int_{N_{2n}(\F)} \phi_J(Ju,1) \theta_{i(y)}^{-1}(u) \, du=1,$$ where $\phi_J$ is the Iwahori fixed vector defined in page 392 of \cite{BFH}. This gives an exact analog to (3.4) in page 395 of \cite{BFH}. From this formula and from the functional equation established already, the explicit formula \eqref{unnor orbits} follows exactly as in pages 395-396 of \cite{BFH}.
\end{proof}
Note that as in \cite{BFH}, the normalizing factors $D(\alpha,y)$ are the values of the Jacquet integrals at the identity.

\section{Unramified Principal series representation of $\mgspn$.} \label{last}

\subsection{A standard model}
Let $\omega$ be a genuine character of $Z \bigl(\mtgspn \bigr)$. We call $\omega$, $\tau(\omega)$ and $I(\omega)$ unramified if the restriction of $\omega$ to $Z \bigl(\mtgspn \bigr) \cap \mkgspn$ is trivial. Note that by Lemma \ref{2 split}, this definition does not depend on $\eta$. From the fourth part of Lemma \ref{ind rep} it follows that if $\omega$ is unramified then exactly two of the four elements of $E(\omega)$ are trivial on $ \mtgsppn \cap \mkgspn$. We call these two extensions, which do not depend on $\eta$, standard extensions. If $(\chi \boxtimes \xi)_\psi$ is one of them we call $i \bigl((\chi \boxtimes \xi)_\psi \bigr)$ and $I \bigl((\chi \boxtimes \xi)_\psi \bigr)$ standard models for $\tau(\omega)$ and $I(\omega)$ respectively. Given that $\psi$ has an even conductor, a genuine character $(\chi \boxtimes \xi)_\psi$ of $\mtgsppn$ has a trivial restriction to $\mtgsppn  \cap \mkgspn$ if and only if it has the form
$$(\chi \boxtimes \xi)_\psi=(\alpha \boxtimes \beta)^+_ \psi={(\chi_\alpha \boxtimes \xi_\beta)}_\psi,$$
where $\beta \in \C^*$ and $\xi_\beta$ is the unramified character of $\F^*$ defined by $x \mapsto \beta^{ord(x)}$.
Note that the restriction of $(\chi_\alpha \boxtimes \xi_\beta)_\psi$ to $\mtspn$ is $(\chi_\alpha)_\psi$. Also note that the two standard extensions  ${(\chi_\alpha \boxtimes \xi_\beta)}_\psi$ and ${(\chi_{\alpha'} \boxtimes \xi_{\beta'})}_\psi$ of $\omega$  satisfy
$${(\chi_{\alpha'} \boxtimes \xi_{\beta'})}_\psi= ^{(i(u_{_0},1))}\! \! {(\chi_\alpha \boxtimes \xi_\beta)}_\psi ={(\chi_{-\alpha} \boxtimes \xi_{(-1)^n\beta})}_\psi.$$
Thus, if $(\alpha \boxtimes \beta)^+_ \psi$ is one of the two standard extensions of an unramified genuine character $\omega$ of $Z\bigl(\mtgspn)$ to $\mtgsppn$ then the two non standard extensions are $$(\alpha \boxtimes \beta)^-_ \psi= ^{(\i(\pi),1)} \! \! \bigl((\alpha \boxtimes \beta)^+_ \psi\bigr)$$ and $(-\alpha \boxtimes (-1)^n\beta)^-_ \psi$.

Remark: Let $\omega$ be a genuine unramified character of $Z \bigl(\mtgspn \bigr)$. Let $E'(\omega)$ be the set of restrictions of the elements of $E(\omega)$ to $\mtspn$. Note that $E'(\omega)$ has four elements. From Section 3.3 of \cite{Sz 5} it follows that the restriction of $I(\omega)$ to $\mspn$ is isomorphic to $$\bigoplus_{\chi \in E'(\omega)} Ind_{\overline{B_{2n}(\F)}}^\mspn \chi.$$
Exactly two of the four summands here have a $\mkspn$ invariant element. These are the two representations induced from restrictions of standard extensions of $\omega$. The other two summands have a $\mkspn^{(i(\pi),1)}$ invariant element. See section 2.6 of \cite{GS}.

\begin{lem}
Let $\omega$ be a genuine character of $Z \bigl(\mtgspn \bigr)$. $\tau(\omega)$ has a non trivial $\mkgspn$ invariant subspace if and only if $\omega$ unramified. In this case the dimension of the invariant subspace is 1. In particular, if $\eta_1\neq \eta_2$ are the two quadratic characters of $\Of^*$ then $\tau(\omega)$ contains a fixed vector under the action of $K^{\eta_1}_{2n}(\F)$ if and only if it contains a fixed vector under the action of $K^{\eta_2}_{2n}(\F)$.
\end{lem}
\begin{proof}
Clearly, the existence of a non-trivial $\mkgspn$ invariant subspace implies that $\omega$ is unramified. Conversely, assume that $\omega$ is unramified and that $i\bigl((\chi \boxtimes \xi)_\psi \bigr)$ is a standard model for $\tau(\omega)$. As a set of representatives of $\mtgsppn \backslash \mtgspn / \mtgspn \cap \mkgspn$ we can take $$\{(I_{2n},1), \, (i(\pi),1)\}.$$ Note that if $f \in i\bigl((\chi \boxtimes \xi)_\psi \bigr)$ is invariant under $\mtgspn \cap \mkgspn$, it follows from the first assertion of Lemma \ref{coc prop} that for any $([u,1],1) \in \mtgsppn \cap \mkgspn$ such that $det(u) \notin {F^*}^2$
$$f(i(\pi),1)=f\bigl((i(\pi),1)([u,1],1)\bigr)=f\bigl(([u,1],-1)(i(\pi),1)\bigr)=-f(i(\pi),1).$$
Thus, $(i(\pi),1)$ can not support a $\mtgspn \cap \mkgspn$ invariant function. On the other hand,
$$f(t)=\begin{cases}(\chi \boxtimes \xi)_\psi(t^{^+})  & t=t^{^+} \! u \, \,  \, where \, \, \,  t^{^+} \in \mtgsppn, \, u \in  \mtgspn \cap \mkgspn  \\ 0  &  otherwise \end{cases}$$
defines a non zero $\mtgspn \cap \mkgspn$ invariant function inside $i\bigl((\chi \boxtimes \xi)_\psi \bigr)$.
\end{proof}

From this Lemma and from the Iwasawa decomposition, it follows that an unramified genuine principal series representation $I(\omega)$ of $\mgspn$ contains a 1 dimensional $\mkgspn$ invariant subspace. If $I\bigl((\chi \boxtimes \xi)_\psi \bigr)$ is a standard model then
\begin{equation} \label{exp sph} f^\eta_{(\chi \boxtimes \xi)_\psi}(g)=\begin{cases} \delta^{\half}(b)(\chi \boxtimes \xi)_\psi(b)  & g=bk, \, where \, \, b\in \mpb, \, k\in \mkgspn \\ 0 & othewise \end{cases}\end{equation}
is the unique $\mkgspn$ invariant function inside $I\bigl((\chi \boxtimes \xi)_\psi \bigr)$ such that $f^\eta_{(\chi \boxtimes \xi)_\psi}(I_{2n},1)=1$.

\subsection{Main result}
Assume that $\psi$ and $\theta$ are as in Section \ref{BFH sec}. For $\alpha \in {\C^*}^n, \, \beta \in {\C^*}$ define two functions $${k^{\eta \, \pm}_{\alpha, \beta}}: \mtgspn \rightarrow \C$$ by

\begin{equation} \label{main def 1} {k^{\eta \, +}_{\alpha, \beta}}(h)=\epsilon \eta \bigl(\lambda(h) \bigr)\gamma_\psi^{-1}(b^n)\beta^l \Delta^{-1}(\alpha)\delta(t)^{\half} \gamma_{\psi}^{-1}(t) \times $$ $$\begin{cases}  A\bigl(\prod_{i=1}^n (1-\eta_\pi(-1) q^{\frac{-1}{2}}\alpha_i^{-1}) \alpha_i^{k_i+i}\bigr) & 0\leq k_1 \leq k_2 \leq \ldots \leq k_n, \, m \!= \! 0\\ 0 & othewise \end{cases} \end{equation} and

\begin{equation} \label{main def 2} {k^{\eta \, -}_{\alpha, \beta}}(h)_{{(\alpha \boxtimes \beta)}_\psi,\pi,\theta}(h)=\epsilon \eta \bigl(\pi\lambda(h) \bigr)\gamma_\psi^{-1}(b^n)\beta^l \Delta^{-1}(\alpha)\delta(t)^{\half} \gamma_{\psi}^{-1}(t) \times $$ $$\begin{cases}  A\bigl(\prod_{i=1}^n \alpha_i^{k_i+i}\bigr) & 0\leq k_1 \leq k_2 \leq \ldots \leq k_n, \, m \!= \! 1\\ 0 & othewise \end{cases}.\end{equation}
Here
\begin{equation} \label{h element} h=\bigl(i(\pi^{-m}),1 \bigr) \bigl(bI_{2n},1 \bigr) \bigl([t,1], 1\bigr) \bigl(i(u),\epsilon \bigr) \end{equation}

where $m \in \{ 0,1 \}, \, b \in \F^*, \, u \in \Of^*$ and $t= diag(a_n,a_{n-1},\ldots,a_1) \in \tgln$ and where we denote $ord(a_i)=k_i, \, ord(b)=l$ (note that $\lambda(h)=b^2u\pi^m$). We now extend these functions to
$${k^{\eta \, \pm}_{\alpha, \beta, \theta}}: \mgspn \rightarrow \C$$ by setting
$${k^{\eta \, \pm}_{\alpha, \beta, \theta}}\bigl((n,1)hu \bigr)=\theta(n){k^\eta_{\alpha, \beta}}^{\pm}(h),$$
where $n \in \zspn, \, u \in \mkgspn$. From the Theorem below it follows that these extensions are well defined. Observe that this definition implies that these are symmetric functions, i.e,
\begin{equation} \label{this is it}{k^{\eta \, \pm}_{\alpha, \beta, \theta}}={k^{\eta \, \pm}_{^w \! \alpha, \beta, \theta}} \end{equation} for any $w \in W_{2n}$.

\begin{thm} \label{main} Let $I(\omega)$ be a genuine unramified principal series representation of $\mgspn$. Suppose
$$E(\omega)=\{ (\alpha \boxtimes \beta)^ {\pm}_ \psi, \, (-\alpha \boxtimes (-1)^n\beta)^ {\pm}_ \psi \}.$$  Then, the set

\begin{equation} \label{main basis} \{ {k^{\eta \, \pm}_{\alpha, \beta, \theta}}, \, {k^{\eta \, \pm}_{-\alpha, (-1)^n\beta, \theta}} \} \end{equation}
form a symmetric spanning set for the space of genuine Spherical $\theta$-Whittaker functions attached to $I(\omega)$.
\end{thm}

\begin{proof} Denote $$Y=\{1,u_{_0},\pi,\pi u_{_0}\}.$$ As explained in Section \ref{gen not},
$Y$ is a set of representatives of
$F^* / {F^*}^2$. For $y \in Y$ define

$$D^{\eta}(\alpha,y)=\eta(y\pi^{-ord(y)})D(\alpha,y)$$
and define

$$W^{\eta}_{{(\alpha \boxtimes \beta)}_\psi,y,\theta}(g)={D^{\eta}(\alpha,y)}^{-1}\int_{N_{2n}(\F)} f^{\eta}_{{(\chi_\alpha \boxtimes \xi_\beta)}_\psi}\bigl((J_{_{2n}}n,1)(i(y),1)g\bigr) \theta_{i(y)}^{-1}(n) \, dn.$$
By Lemma \ref{basis lemma again}. $\{W^{\eta}_{{(\alpha \boxtimes \beta)}_\psi,y,\theta} \mid y \in Y \}$
is a spanning set for the space of genuine Spherical $\theta$-Whittaker functions attached to $I(\omega)$. Thus, the proof of this Theorem amounts to showing  the following:
\begin{equation} \label{main 1} W^{\eta}_{{(\alpha \boxtimes \beta)}_\psi,1,\theta}(h)={k^{\eta \, +}_{\alpha,\beta, \theta}}(h), \end{equation}

\begin{equation} \label{main 2} W^{\eta}_{{(\alpha \boxtimes \beta)}_\psi,\pi,\theta}(h)={k^{\eta \, -}_{\alpha, \beta, \theta}}(h), \end{equation}

\begin{equation} \label{main 3} W^{\eta}_{{(\alpha \boxtimes \beta)}_\psi,u_{_0},\theta}=W^{\eta}_{{(\chi_{-\alpha} \boxtimes \xi_{(-1)^n\beta})}_\psi,1\theta}\end{equation}

and \begin{equation} \label{main 4} W^{\eta}_{{(\alpha \boxtimes \beta)}_\psi,\pi u_{_0},\theta}=W^{\eta}_{{(\chi_{-\alpha} \boxtimes \xi_{(-1)^n\beta})}_\psi,\pi,\theta}.\end{equation}
Here $h$ is the $\mtgspn$ element defined in \eqref{h element}.

By \eqref{exp sph}, $f^\eta_{(\chi_\alpha \boxtimes \xi_\beta)_\psi}(g)$  vanishes if $\lambda(g) \not \in \Of^* {\F^*}^2$. This implies that
$W^{\eta}_{{(\alpha \boxtimes \beta)}_\psi,y,\theta}(h)$ vanish if $m \neq ord(y)$. We compute $W^{\eta}_{{(\alpha \boxtimes \beta)}_\psi,1,\theta}(h)$. Recall that $Z \bigl(\mgpspn \bigr)= \overline{\F^*I_{2n}}$ and that
$$f^{\eta}_{{(\chi_\alpha \boxtimes \xi_\beta)}_\psi}\bigl(g (i(k),\epsilon) \bigr)=\epsilon \eta(k)f^{\eta}_{{(\chi_\alpha \boxtimes \xi_\beta)}_\psi}(g).$$
Thus, if $m=0$,
$$W^{\eta}_{{(\alpha \boxtimes \beta)}_\psi,1,\theta}(h)=\epsilon\eta(u)\gamma_\psi^{-1}(b^n)\beta^l W^{\eta}_{{(\alpha \boxtimes \beta)}_\psi,1,\theta}\bigl(t,1 \bigr).$$
Since the restriction of $W^{\eta}_{{(\alpha \boxtimes \beta)}_\psi,1,\theta}$ to $\mspn$ is  $W^0_{\alpha,\psi,1,\theta}$, \eqref{main 1} now follows from \eqref{nor orbits}. To compute $W^{\eta}_{{(\alpha \boxtimes \beta)}_\psi,u_{_0},\theta}(h)$ we note that
$$\bigr(i(u_{_0}),1\bigl)h\bigr(i(u_{_0}),1\bigl)^{-1}=\bigl(I_{2n}, (u_{_0},b^na_1,a_2\ldots a_n)\f \bigr)h $$
and that $$(u_{_0},b^na_1,a_2\ldots a_n)\f={(-1)}^{(ln+k_1+k_2\ldots +k_n)}.$$
Thus, arguing as before, if $m=0$,
$$W^{\eta}_{{(\alpha \boxtimes \beta)}_\psi,u_{_0},\theta}(h)=
{(-1)}^{(ln+k_1+k_2\ldots +k_n)}\eta(u)\gamma_\psi^{-1}(b^n)\beta^l \Delta^{-1}(\alpha)\delta(t)^{\half}\gamma_\psi^{-1}(t)$$
$$\begin{cases}  A\bigl(\prod_{i=1}^n (1-\eta_\pi(-u_{_0}) q^{\frac{-1}{2}}\alpha_i^{-1}) \alpha_i^{k_i+i}\bigr) & 0\leq k_1 \leq k_2 \leq \ldots \leq k_n \\ 0 & othewise \end{cases}.$$
\eqref{main 3} follows now from \eqref{unram quad} and  \eqref{to be used again}. \eqref{main 2} and \eqref{main 4} are proven by similar arguments which ultimately utilize \eqref{unnor orbits}.
\end{proof}
Note that from the proof above it follows that for $y \in Y.$
$$D^{\eta}({\alpha},y)=\int_{N_{2n}(\F)} f^{\eta}_{{(\chi_\alpha \boxtimes \xi_\beta)}_\psi}\bigl((J_{_{2n}}n,1)(i(y\pi^{-ord(y)}),1)g\bigr) \theta_{y}^{-1}(n) \, dn.$$

\begin{thm} \label{Final}If $n$ is odd then the spanning set constructed in Theorem \ref{main} has an additional property: Let $\omega'$ be the central character of $I(\omega)$. $\omega'$ is the restriction of $\omega$ to $Z \bigl(\mgspn \bigr)$. Define $\Omega$ to be the set of four extensions of $\omega'$ to $Z \bigl(\mgsppn \bigr)$. For each $\mu$ in $\Omega$ exactly one of the functions in \eqref{main basis} has the property
$$f(tg)=\mu(t)f(g)$$ for any $t \in Z \bigl(\mgsppn \bigr), \, g \in \mgspn$
\end{thm}

\begin{proof}  This is verified at once by direct computation using the cocycle properties given in Lemma \ref{coc prop}. In fact it follows from Corollary \ref{n odd stru}.
\end{proof}
\section{Reducible unramified principal representations} \label{app}
Fix $\psi$, a normalized character of $\F$. Through this section, let $\omega$ be an unramified genuine character of $Z \bigl(\mtgspn \bigr)$ and let $I(\omega)$ be an unramified principal series representation of $\mgspn$ with $I\bigl((\chi_\alpha \boxtimes \xi_\beta)_\psi\bigr)$ and $I\bigl((\chi_{-\alpha} \boxtimes \xi_{(-1)^n\beta})_\psi\bigr)$ as its two standard models.
\begin{lem} \label{r 2 unram}$R(\omega)>1$ if and only if $n$ is even and up to conjugating by elements of $W_{2n}$, $\alpha$ equals
\begin{equation} \label{special alpha} \alpha_{_0}= (\alpha_1, -\alpha_1,\alpha_2, -\alpha_2,\ldots,\alpha_{\frac n 2}, -\alpha_{\frac n 2}) \end{equation}
In this case $R(\omega)=2$.
\end{lem}
\begin{proof} Since $W_{2n}$ maps a standard extension of an unramified genuine charter of $Z \bigl(\mtgspn \bigr)$ to another standard extension of an unramified genuine charter of $Z \bigl(\mtgspn \bigr)$ it follows that $R(\omega) \leq 2$. The rest of this Lemma follows from Lemma \ref{act on E}.
\end{proof}
\begin{lem} \label{short and nice}If $R(\omega)>1$ then the dimension of the space of Whittaker functionals on the space generated by the $\mkgspn$ invariant space inside $I(\omega)$ is 2. In particular, $I(\omega)$ has more then one generic constituent.
\end{lem}
\begin{proof}
By Lemma \ref{r 2 unram}, $R(\omega)>1$ implies that $n$ is even and that $^w \! \alpha=-\alpha$ for some $w \in W_{2n}$. Thus, by  \eqref{this is it}
$${k^{\eta \, +}_{\alpha, \beta}}={k^{\eta \, +}_{-\alpha, \beta}}, \, \, \, \,{k^{\eta \, -}_{\alpha, \beta}}={k^{\eta \, -}_{-\alpha, \beta}}.$$
On the other hand ${k^{\eta \, +}_{\alpha, \beta}}$ and ${k^{\eta \, -}_{\alpha, \beta}}$ are linearly independent since these two functions have disjoint supports. Hence, the lemma  follows from the symmetry property of the spanning set constructed in Theorem \ref{main}.
\end{proof}
From this point we assume that $n$ is even. For a character $\chi$ of $\tspn$ and  $$\bold{s}=(s_1,s_2,...s_n)\in {\C^*}^n$$ let $\chi^{\bold{s}}$ be the character of  $\tspn$ given by
$$t=[diag(t_n,t_{n-1},\ldots,t_1),1] \mapsto \chi(t) \prod_{i=1}^n \ab t_i \ab^{s_i}.$$
Let $w_{_0}$ be the $W_{2n}$ element defined  in \eqref{special w}. For $s=(s_1,s_2,...s_n)\in {\C^*}^n$ define and $f \in I\bigl((\chi^{\bold{s}})_\psi\bigr)$ define a complex function $A_{\chi,s}(f)=A_s(f)$ on $\mspn$ by
\begin{equation} \label{my inter}\bigl(A_s(f)\bigr)(g)=\int_{\zspn \cap w_{_0} N_{2n}^-(\F)w_{_0}^{-1}} f\bigl((w_{_0}n,1)g\bigr) \, dn.\end{equation}
Here $N_{2n}^-(\F)=J_{2n} \zspn J_{2n}^{-1}$ is the unipotent radical of $\mspn$ opposite to $\zspn$.

\begin{lem} \label{inter my self} (1) The integral defined in \eqref{my inter} converge in some cone inside ${\C^*}^n$ and has a meromorphic extension to ${\C^*}^n$. Away from its polls it defines an $\mspn$ intertwining map
$$A_s:I\bigl((\chi^{\bold{s}})_\psi\bigr) \rightarrow I\bigl(^{w_{_0}} \! (\chi^{\bold{s}})_\psi\bigr).$$
(2) Let $\chi_{_0}$ be as in \eqref{r not tri}. Then $A_{\chi_{_0},s}$ is analytic  at $s=0$, i.e, at $\chi_{_0}$.\\
(3) Let $\chi_{_0}=\chi_{\alpha_0}$ where $\alpha_0$ is as in \eqref{special alpha} then
$$A_{\chi_{\alpha_0},s}\bigl(f^0_{(\chi_{\alpha_0})_\psi}\bigr)=\Bigl( \frac {L(\eta_{u_0},0)}{L(\eta_{u_0},1)}\Bigr)^{\frac n 2} \bigl(f^0_{(\chi_{-\alpha_0})_\psi}\bigr)$$
\end{lem}
\begin{proof} (1) is well known. It is proven by the same standard arguments used for linear groups. To prove (2) we may assume the $g$ in \eqref{my inter} is $(I_{2n},1)$. Also by a standard argument, we decompose $A_s$
to a rank one intertwining operators:

$$A_s=A_s^{\frac n 2} \circ A_s^{\frac {n-1} {2}} \circ \dots \circ A_s^1 $$ where
$$A^j_s:I\bigl(^{w_{j-1}w_{j-2}\ldots{w_1}} \! (\chi_{_0}^{\bold{s}})_\psi\bigr) \rightarrow I\bigl(^{w_{j}w_{j-1}\ldots{w_1}} \! (\chi_{_0}^{\bold{s}})_\psi\bigr).$$ is defined by (the meromorphic continuation of)

$$\bigl(A^j_s (f) \bigr)(g)=\int_{\F} f\bigl((w^j_{_0}n^j(x),1)g\bigr) \, dx.$$
Here $$w^j_{_0}=diag \bigl(I_{n-2j},w',I_{2(j-1)},I_{n-2j},w',I_{2(j-1)}\bigr),$$
where $w'=\begin{pmatrix} _{0} &
_{1}\\_{1} & _{0} \end{pmatrix}$ and $$n^j(x)=diag \bigl(I_{n-2j},n'(x),I_{2(j-1)},I_{n-2j},^t n(-x),I_{2(j-1)}\bigr),$$
where $n'(x)=\begin{pmatrix} _{1} &
_{x}\\_{0} & _{1} \end{pmatrix}.$
Since the $\C^1$ cover of $\spn$ splits over the Siegel parabolic subgroup, the proof of (2) and (3) is now reduced to the well known $GL_2(\F)$ computations. More precisely, by Rao cocycle Formula given in Theorem 5.3 of \cite{R},
\begin{equation} \label{rao rao} c\Bigr(\begin{pmatrix} _{A} &
_{B}\\_{0} & _{^{t}A^{-1}} \end{pmatrix}, \begin{pmatrix} _{A'} &
_{B'}\\_{0} & _{^{t}A'^{-1}} \end{pmatrix} \Bigr)=\bigl(det(A),det(A')\bigr)\f. \end{equation}
For $f \in I\bigl((\chi^{\bold{s}})_\psi \bigr)$ and $g \in GL_2(\F)$ define
$$\bigl(C_j(f)\bigr)(g)=\gamma_\psi(g)\ab det(g) \ab^{\half-2j}f\bigl(I_{n-2j},g,I_{2(j-1)},I_{n-2j},^t g ^{-1},I_{2(j-1)}\bigr).$$
From \eqref{gammaprop} and \eqref{rao rao} it follows that
$$C^j(f)\in Ind^{GL_2(\F)}_{B(\F)} (\chi^{\bold{s}})^j.$$
Here $B(\F)$ is the standard Borel subgroup of $GL_2(\F)$ and $$ (\chi^{\bold{s}})^j(t_2,t_1)=\chi^{\bold{s}}\bigl(I_{n-2j},t_2,t_1,I_{2(j-1)},I_{n-2j},t_2^{-1},t_1^{-1},I_{2(j-1)}\bigr).$$
(the term $ \ab det(g) \ab^{\half-2j}$ appears here to balance the difference between the modular functions of $B_{2n}(\F)$ and $B(\F)$). The point is that for $f \in I\bigl((\chi^{\bold{s}})_\psi \bigr)$,

$$\bigl(A_{\chi^{\bold{s}}}^j (f)\bigr)(I_{2n},1)=\bigl(A_{(\chi^{\bold{s}})^j} (C_j(f)\bigr)(I_2),$$
where $A_{(\chi^{\bold{s}})^j}$ is the standard $GL_2(\F)$ intertwining integral, see page 478 of \cite{B} for example. This completes the reduction to $GL_2(F)$: (2) follows now from Proposition 4.5.7 of \cite{B} and (3) follows from Proposition 4.6.7 of \cite{B}.

\end{proof}
\begin{cor} \label{this is the point} For $f \in I\bigl((\chi_{\alpha_{0}} \boxtimes \xi_\beta)_\psi \bigr)$ and $g \in \mgspn$, define $\bigl(A(f)\bigr)(g)$ to be the meromorphic continuation of $$\bigl(A_s(f)\bigr)(g)=\int_{\zspn \cap w_{_0} N_{2n}^-(\F)w_{_0}^{-1}} f\bigl((w_{_0}n,1)g\bigr) \, dn.$$ The map $$A:I\bigl((\chi_{\alpha_0} \boxtimes \xi_\beta)_\psi\bigr) \rightarrow I\bigl((\chi_{-\alpha_0} \boxtimes \xi_\beta)_\psi\bigr)$$
is a well defined $\mgspn$ intertwining map and $$A\bigl(f^{\eta}_{(\chi_{\alpha_0} \boxtimes \xi_\beta)_\psi} \bigr)=\Bigl( \frac {L(\eta_{u_0},0)}{L(\eta_{u_0},1)}\Bigr)^{\frac n 2}f^{\eta}_{(\chi_{-\alpha_0} \boxtimes \xi_\beta)_\psi}.$$
\end{cor}
\begin{proof}
First, by using the fact that $\mgsppn= Z(\mgsppn)\mspn$, we extend $A_s$ to be a $\mgsppn$ intertwining map from $I'\bigl((\chi_{\alpha_0} \boxtimes \xi_\beta)_\psi\bigr)$ to $I'\bigl((\chi_{-\alpha_0} \boxtimes \xi_\beta)_\psi\bigr)$ (the notation $I'$ was defined in \eqref{i prime}). Then, using the same induction by stages argument used in the beginning of  proof of Theorem \ref{basic red}, we push $A_s$ to be the $\mgspn$ intertwining map defined in this corollary. Part (2) of Lemma \ref{inter my self} grantees that  $A$ is holomorphic at $(\chi_{\alpha_0} \boxtimes \xi_\beta)_\psi$.

Since the $\mgspn$ invariant space inside $I\bigl((\chi_\alpha \boxtimes \xi_\beta)_\psi\bigr)$ is one dimensional it follows that
$$A\bigl(f^{\eta}_{(\chi_{\alpha_0} \boxtimes \xi_\beta)_\psi} \bigr)=c_{\alpha_0}f^{\eta}_{(\chi_{-\alpha_0} \boxtimes \xi_\beta)_\psi},$$
where $c_{\alpha_0}$ is the analytic continuation of
$$\int_{\zspn \cap w_{_0} N_{2n}^-(\F)w_{_0}^{-1}} f^{\eta}_{(\chi_{\alpha_0} \boxtimes \xi_\beta)_\psi} (w_{_0}n)  \, dn. $$
Note that since the restriction of both $f^{\pm}_{(\chi_\alpha \boxtimes \xi_\beta)_\psi}$to $\mspn$ is $f^{0}_{(\chi_\alpha)_\psi}$, it follows that $c_{\alpha_0}$ does not depend on $\eta$. In fact, by (3) of Lemma \ref{inter my self}
$$c_{\alpha_0}=\Bigl( \frac {L(\eta_{u_0},0)}{L(\eta_{u_0},1)}\Bigr)^{\frac n 2} \neq 0.$$
\end{proof}
Recall that $\eta_1$ and $\eta_\pi$ are respectively the trivial and non-trial quadratic characters of $\Of^*$. We shall denote
$$\mkgspnp=K^{\eta_1}_{2n}(\F), \, \, \, \mkgspnm=K^{\eta_{u_{0}}}_{2n}(\F).$$
We shall also denote by $I^+(\omega)$ and $I^-(\omega)$ the (isomorphism classes) of the sub representations of $I(\omega)$ generated by the $\mkgspnp$ and $\mkgspnm$ invariant subspaces.

\begin{thm} \label{the end}Assume that  $I \bigl( (\chi_\alpha)_\psi \bigr)$ is an irreducible $\mspn$ module. Then, $I(\omega)$ is reducible if and only if  $R(\omega)>1$.  In this case $I^+(\omega)$ and $I^-(\omega)$ are irreducible and $$I(\omega) \simeq I^+(\omega) \oplus I^-(\omega)$$ is a decomposition of $I(\omega)$ into a direct sum of two irreducible non-isomorphic generic subspaces, each has a space of Whittaker functionals of dimension 2.
\end{thm}
\begin{proof} Given Theorem \ref{basic red} and Lemma \ref{short and nice} we only need to show that $I^+(\omega) \neq I^-(\omega)$. Since $I \bigl( (\chi_\alpha)_\psi \bigr)$ is irreducible we may conjugate the inducing character
by $w \in W_{2n}$ without changing the isomorphism class of $I\bigl( (\chi_\alpha)_\psi \bigr)$. Hence, since $R(\omega)>1$ it follows from Lemma \ref{special alpha} that we may assume that $\alpha=\alpha_{_0}$. Denote $f^+_{(\chi_{\alpha} \boxtimes \xi_\beta)_\psi}=f^{\eta_1}_{(\chi_{\alpha} \boxtimes \xi_\beta)_\psi}$ and $f^-_{(\chi_{\alpha} \boxtimes \xi)_\psi}=f^{\eta_\pi}_{(\chi_{\alpha} \boxtimes \xi)_\psi}$. We shall construct here a self intertwining map
$T$ on $I\bigl( (\! \chi_{\alpha} \boxtimes \xi_\beta)_\psi\bigr)$ and show that $f^{\pm}_{(\chi_{\alpha} \boxtimes \xi_\beta)_\psi}$ are eigen values of $T$ corresponding to two different eigen values. We define $$B:I\bigl((\chi_{-\alpha} \boxtimes \xi_\beta)_\psi\bigr) \rightarrow I\bigl((\chi_{\alpha} \boxtimes \xi_\beta)_\psi\bigr)$$ by
$$B \bigl(f \bigr)(g)=f\bigl((i(u_{_0},1)g\bigr).$$
By the same argument as in Corollary \ref{this is the point}
$$B\bigl(f^{\eta}_{(\chi_{-\alpha} \boxtimes \xi_\beta)_\psi} \bigr)=d_\eta f^{\eta}_{(\chi_{\alpha} \boxtimes \xi_\beta)_\psi}$$
where this time $d_\eta$ does depend on $\eta$ and is defined by
$$d_\eta=f^{\eta}_{(\chi_{-\alpha} \boxtimes \xi_\beta)_\psi}\bigl((i(u_{_0},1)\bigr)=\eta(u_{_0}).$$
Let $A$ be the intertwining map defined in  Corollary \ref{this is the point}.$$B \circ A: I\bigl((\chi_{\alpha} \boxtimes \xi_\beta)_\psi\bigr) \rightarrow I\bigl((\chi_{\alpha} \boxtimes \xi_\beta)_\psi\bigr)$$ is a self intertwining operator and by Corollary \ref{this is the point}

$$B \circ A \bigl(f^+_{(\chi_\alpha \boxtimes \xi_\beta)_\psi} \bigr)=\Bigl( \frac {L(\eta_{u_0},0)}{L(\eta_{u_0},1)}\Bigr)^{\frac n 2}f^+_{(\chi_\alpha \boxtimes \xi_\beta)_\psi}, \, \, \, \,
B \circ A \bigl(f^-_{(\chi_\alpha \boxtimes \xi_\beta)_\psi} \bigr)=-\Bigl( \frac {L(\eta_{u_0},0)}{L(\eta_{u_0},1)}\Bigr)^{\frac n 2}f^-_{(\chi_\alpha \boxtimes \xi_\beta)_\psi}.$$

\end{proof}
Recall that if the restriction of $\omega$ to $Z\bigl(\mtgspn \bigr) \cap \mspn$ is unitary then  $I \bigl( (\chi_\alpha)_\psi \bigr)$ is irreducible. Hence, we have proven:
\begin{cor} All the reducible unramified principal series representations of $\mgspn$ induced from a unitary data are a direct some of two non-isomorphic irreducible generic subspaces. Each has a 2 dimensional space of Whittaker functionals.
\end{cor}

Department of Mathematics, Purdue University.\\
email: dszpruch@math.purdue.edu 

\end{document}